\documentclass[a4paper,10pt,reqno]{amsart}
\usepackage{bm,amsmath,amsfonts,amscd,amsthm,amssymb,graphicx,mathrsfs,mathrsfs,tikz,tikz-cd,comment,mathtools,mathdots,todonotes,stmaryrd,upgreek,euscript,colonequals}
\usepackage[nospace,noadjust]{cite}
\usepackage[pagebackref]{hyperref}

\hypersetup{
  colorlinks   = true,          
  urlcolor     = blue,          
  linkcolor    = purple,          
  citecolor   = blue             
}

\usepackage{enumitem}
\setlist[enumerate]{format=\normalfont}

\makeatletter
\let\save@mathaccent\mathaccent
\newcommand*\if@single[3]{%
  \setbox0\hbox{${\mathaccent"0362{#1}}^H$}%
  \setbox2\hbox{${\mathaccent"0362{\kern0pt#1}}^H$}%
  \ifdim\ht0=\ht2 #3\else #2\fi
  }
\newcommand*\rel@kern[1]{\kern#1\dimexpr\macc@kerna}
\newcommand*\widebar[1]{\@ifnextchar^{{\wide@bar{#1}{0}}}{\wide@bar{#1}{1}}}
\newcommand*\wide@bar[2]{\if@single{#1}{\wide@bar@{#1}{#2}{1}}{\wide@bar@{#1}{#2}{2}}}
\newcommand*\wide@bar@[3]{%
  \begingroup
  \def\mathaccent##1##2{%
    \let\mathaccent\save@mathaccent
    \if#32 \let\macc@nucleus\first@char \fi
    \setbox\z@\hbox{$\macc@style{\macc@nucleus}_{}$}%
    \setbox\tw@\hbox{$\macc@style{\macc@nucleus}{}_{}$}%
    \dimen@\wd\tw@
    \advance\dimen@-\wd\z@
    \divide\dimen@ 3
    \@tempdima\wd\tw@
    \advance\@tempdima-\scriptspace
    \divide\@tempdima 10
    \advance\dimen@-\@tempdima
    \ifdim\dimen@>\z@ \dimen@0pt\fi
    \rel@kern{0.6}\kern-\dimen@
    \if#31
      \overline{\rel@kern{-0.6}\kern\dimen@\macc@nucleus\rel@kern{0.4}\kern\dimen@}%
      \advance\dimen@0.4\dimexpr\macc@kerna
      \let\final@kern#2%
      \ifdim\dimen@<\z@ \let\final@kern1\fi
      \if\final@kern1 \kern-\dimen@\fi
    \else
      \overline{\rel@kern{-0.6}\kern\dimen@#1}%
    \fi
  }%
  \macc@depth\@ne
  \let\math@bgroup\@empty \let\math@egroup\macc@set@skewchar
  \mathsurround\z@ \frozen@everymath{\mathgroup\macc@group\relax}%
  \macc@set@skewchar\relax
  \let\mathaccentV\macc@nested@a
  \if#31
    \macc@nested@a\relax111{#1}%
  \else
    \def\gobble@till@marker##1\endmarker{}%
    \futurelet\first@char\gobble@till@marker#1\endmarker
    \ifcat\noexpand\first@char A\else
      \def\first@char{}%
    \fi
    \macc@nested@a\relax111{\first@char}%
  \fi
  \endgroup
}
\makeatother

\makeatletter
\renewcommand{\@marginparreset}{%
  \reset@font\scriptsize
  \raggedright
  \@setminipage
}
\makeatother

\usetikzlibrary{arrows,decorations.pathmorphing,decorations.pathreplacing,positioning,shapes.geometric,shapes.misc,decorations.markings,decorations.fractals,calc,patterns}

\tikzset{
        cvertex/.style={circle,draw=black,inner sep=1pt,outer sep=3pt},
        vertex/.style={circle,fill=black,inner sep=1pt,outer sep=3pt},
        star/.style={circle,fill=yellow,inner sep=0.75pt,outer sep=0.75pt},
        tvertex/.style={inner sep=1pt,font=\scriptsize},
        gap/.style={inner sep=0.5pt,fill=white}}

\newtheorem{thm}{Theorem}[section]
\newtheorem{prop}[thm]{Proposition}
\newtheorem{lemma}[thm]{Lemma}
\newtheorem{cor}[thm]{Corollary}
\newtheorem{assum}[thm]{Assumption}

\theoremstyle{definition} 
\newtheorem{defin}[thm]{Definition}

\newtheorem{example}[thm]{Example}
\newtheorem{setup}[thm]{Setup}

\newtheorem{remark}[thm]{Remark}

\newtheorem{notation}[thm]{Notation}
\newtheorem{construction}[thm]{Construction}
\newtheorem{property}[thm]{Property}

\numberwithin{equation}{section}

\DeclareMathAlphabet{\mathbbm}{U}{bbm}{m}{n}

\newcommand{\rmA}{\mathrm{A}}
\newcommand{\Acon}{\rmA_{\textnormal{con}}}
\newcommand{\Bcon}{\mathrm{B}_{\textnormal{con}}}
\newcommand{\Ccon}{\mathrm{C}_{\textnormal{con}}}

\newcommand{\DG}{{\textnormal{DG}}}

\newcommand{\Ob}{\operatorname{Ob}}

\newcommand{\Hom}{\textnormal{Hom}}

\newcommand{\HH}{\textnormal{H}}
\newcommand{\Ext}{\textnormal{Ext}}
\newcommand{\Tor}{\textnormal{Tor}}

\newcommand{\End}{\textnormal{End}}

\newcommand{\RHom}{\textnormal{RHom}}

\def\Spec{\mathop{\rm Spec}\nolimits}

\def\mod{\mathop{\textnormal{mod}}\nolimits}
\def\coh{\mathop{\textnormal{coh}}\nolimits}

\def\Ker{\mathop{\rm Ker}\nolimits}

\def\Im{\mathop{\rm Im}\nolimits}

\newcommand{\image}{\operatorname{Im}}
\newcommand{\Id}{\operatorname{Id}}

\let\oldtocsection=\tocsection
\let\oldtocsubsection=\tocsubsection
\let\oldtocsubsubsection=\tocsubsubsection
\renewcommand{\tocsection}[2]{\hspace{0em}\oldtocsection{#1}{#2}}
\renewcommand{\tocsubsection}[2]{\hspace{1em}\oldtocsubsection{#1}{#2}}
\renewcommand{\tocsubsubsection}[2]{\hspace{2em}\oldtocsubsubsection{#1}{#2}}

\newcommand{\scrA}{\EuScript{A}}
\newcommand{\scrB}{\EuScript{B}}

\newcommand{\scrE}{\EuScript{E}}

\newcommand{\scrM}{\EuScript{M}}

\newcommand{\scrN}{\EuScript{N}}
\newcommand{\scrO}{\EuScript{O}}

\newcommand{\scrP}{\EuScript{P}}
\newcommand{\scrQ}{\EuScript{Q}}

\newcommand{\scrT}{\EuScript{T}}

\newcommand{\scrV}{\EuScript{V}}

\def\scrEnd{\mathop{\EuScript{E}\rm nd}\nolimits}
\def\scrHom{\mathop{\EuScript{H}\rm om}\nolimits}

\newcommand{\sigE}{{^{\upsigma}\kern -1pt\scrE}}
\newcommand{\nsigE}{{^{{\scriptstyle-}\upsigma}\kern -1pt\scrE}}

\def\Db{\mathop{\rm{D}^b}\nolimits}

\numberwithin{equation}{section}

\newcommand{\PP}{\mathrm{P}}

\makeatletter
\newcommand*\bigcdot{\mathpalette\bigcdot@{.5}}
\newcommand*\bigcdot@[2]{\mathbin{\vcenter{\hbox{\scalebox{#2}{$\m@th#1\bullet$}}}}}
\makeatother

\begin{document}
\title[Trivial extension and unitally positive algebras]{Trivial extension DG-algebras, unitally positive $A_\infty$-algebras, and applications}
\author{Joseph Karmazyn}
\author{Emma Lepri}
\author{Michael Wemyss}
\begin{abstract}
To any periodic module over any algebra, this paper introduces an associated trivial extension $\DG$-algebra $\scrT$.  After first passing to a strictly unital $A_\infty$ minimal model, it then constructs a particular $A_\infty$-algebra $\scrN$, called the unitally positive $A_\infty$-algebra, which roughly speaking describes the identity in degree zero and all the positive cohomology.  The object $\scrN$ is fundamental, and can be constructed for any $\DG$-category satisfying very mild assumptions.

The main application is to birational geometry.  When applied to contraction algebras, the construction gives a simple and direct proof of the Donovan--Wemyss conjecture, namely that smooth irreducible $3$-fold flops are classified by their contraction algebras, and thus by noncommutative data. 
\end{abstract}
\maketitle

\section{Introduction}

Surgeries that preserve the birational class of a variety are fundamental to algebraic geometry.  In dimension two, contractions of ($-1$)-curves in smooth surfaces were known to the Italian School in the 1890s. Dimension three is much harder: it took a further eighty years, at the advent of the modern minimal model programme in the late 1970s, for the next fundamentally new surgeries to be discovered, namely smooth $3$-fold flops.  

Over 45 years on from that discovery, this paper is devoted to giving a simple and direct  proof of the conjecture \cite{DW1}, which predicted that smooth $3$-fold flops can in fact be classified, using noncommutative data.  The novelty in this paper is to prove the conjecture by constructing and controlling two new (derived) objects: first the trivial extension $\DG$-algebra $\scrT$, then a particular `unitally positive' $A_\infty$-algebra $\scrN$. 

Happily, $\scrT$ can be constructed in very general periodic settings, and the main character $\scrN$ can be defined for very general $\DG$-categories. As such, neither construction is specific to birational geometry, and we so begin by outlining both in detail.

\subsection{The trivial extension $\DG$-algebra}\label{subsec: intro TEDGA} Consider any algebra $\Upgamma$, and any $M\in\mod\Upgamma$ such that there exists a complex of finitely generated projective $\Upgamma$-modules
$\PP=(P_{n-1} \to P_{n-2} \to \hdots \to P_{1} \to P_0)$ where
\[
0 \to M \xrightarrow{\upbeta}  P_{n-1} \to P_{n-2} \to \dots \to P_{1} \to P_0 \xrightarrow{\upalpha} M  \to 0
\]
is exact.  In other words, consider any periodic $\Upgamma$-module $M$.  The complex $\PP$ is said to have length~$n$. Mildly abusing notation, by shifting the complex $\PP $ and using the connecting map $d_0\colonequals \upbeta\circ\upalpha$, consider next the complex
\[
\scrP\colonequals\quad \dots\to  \PP [2n] \xrightarrow{d_0[n]} \PP [n] \xrightarrow{d_0} \PP  \to 0
\]
where each component $\PP [n\ell]$ represents a length $n$ complex of projective $\Upgamma$-modules.  The complex $\scrP$ gives a (possibly non-minimal) projective resolution of $M$.  From this, consider the usual $\DG$-algebra 
$\scrE\colonequals \scrEnd_{\Upgamma}(\scrP)$,
and the element $\upsigma\in \scrE$ defined as
\[
\upsigma\colonequals
\begin{tikzcd}
\phantom{\PP [8]}  \arrow{r}\arrow{rd}{\Id} & \PP [2n] \arrow{r} \arrow{rd}{\Id} & \PP [n] \arrow{r} \arrow{rd}{\Id} & \PP  \arrow{r} \arrow{rd}{0} & 0 \arrow{r} \arrow{rd}{0} & \dots\\
\phantom{\PP [8]}  \arrow{r}  & \PP [2n] \arrow{r}  & \PP [n] \arrow{r} & \PP \arrow{r}  & 0 \arrow{r} & \dots
\end{tikzcd}
\]
Write $\sigE\colonequals \bigoplus_{i\in\mathbb{Z}}\sigE^i$, where
$\sigE^i \colonequals \{ x \in \scrE^i \mid \upsigma \circ x =(-1)^{|\upsigma||x|} x \circ \upsigma \}$ and $\scrE^i$ is the $i$th graded piece of $\scrE$, recalled in \ref{key setup}.  There is a similar version $\nsigE$, which accounts for an additional sign; see \ref{def nSig} for details.  

The trivial extension $\DG$-algebra $\scrT$ is then defined as follows.  As a graded vector space, $\scrT \colonequals \sigE \oplus \nsigE[1-n]$, whilst multiplication is defined on homogeneous elements by
\[
(x,y) \cdot (a,b) \colonequals  \big( x \circ a, \, x \circ b + (-1)^{|a|(|\upsigma|+1)} y \circ a \big).
\] 
Up to ignoring signs and shifts, the above is precisely the multiplication on the classical trivial extension algebra. The main point here, and perhaps the main novelty, is that with the signs and shifts, there exists a differential $\upxi$ on $\scrT$ defined by 
\[
\upxi (x,y) \colonequals \big(\updelta(x) - (-1)^{|x|} y \circ \upsigma, \,\updelta(y)\big).
\]
We explain in \ref{rem:actually_a_cone} that $\upxi$ can be viewed as the differential on a particular cone.

Our first result asserts that the multiplication and differential are compatible.
\begin{prop}[\ref{Triv DG is DG}]\label{Triv DG is DG intro}
$(\scrT,\cdot,\upxi)$ is a $\DG$-algebra.
\end{prop}

Constructing $\DG$-algebras which are simultaneously explicit and useful is surprisingly difficult.  One way of interpreting \ref{Triv DG is DG intro} is that it gives a new way of constructing a $\DG$-algebra, whenever one is faced with some form of periodic behaviour.

One of the features of the $\DG$-algebra $\scrT$ is that it has cohomology in both positive and negative degrees.
\begin{prop}[\ref{eq.coh-t}]\label{eq.coh-t_intro}
If the natural map $\Hom_{\Upgamma}(M,M) \to \Ext^{n}_{\Upgamma}(M,M)$ in \textnormal{\ref{lem.map-pi}} is injective,  then the only non-zero cohomology groups of $\scrT$ are the following, where in each degree take the direct sum of the groups stated.
\[
\begin{tikzpicture}
\node (a2) at (-4.6,-0.2) {\scalebox{0.8}{$\HH^{-(n-1)}(\scrT)$}};
\node (a3) at (-3.35,-0.3) {\scalebox{0.8}{$\hdots$}};
\node (a4) at (-2.1,-0.2) {\scalebox{0.8}{$\HH^{-1}(\scrT)$}};
\node (a5) at (0,-0.2) {\scalebox{0.8}{$\HH^{0}(\scrT)$}};
\node (a6) at (2,-0.2) {\scalebox{0.8}{$\HH^{1}(\scrT)$}};
\node (a7) at (3.1,-0.3) {\scalebox{0.8}{$\hdots$}};
\node (a8) at (4.4,-0.2) {\scalebox{0.8}{$\HH^{n-2}(\scrT)$}};
\node (a8) at (6.5,-0.2) {\scalebox{0.8}{$\HH^{n-1}(\scrT)$}};

\node (b2) at (-4.6,-0.95) {\scalebox{0.8}{$\Hom_{\Upgamma}(M,M)$}};
\node (b3) at (-3.35,-1.05) {\scalebox{0.8}{$\hdots$}};
\node (b4) at (-2.1,-0.95) {\scalebox{0.8}{$\Ext^{n-2}_{\Upgamma}(M,M)$}};
\node (b5) at (0,-0.95) {\scalebox{0.8}{$\Ext^{n-1}_{\Upgamma}(M,M)$}};
\node (b5c) at (0,-1.5) {\scalebox{0.8}{$\Hom_\Upgamma(M,M)$}};
\node (b6) at (2,-1.5) {\scalebox{0.8}{$\Ext^1_\Upgamma(M,M)$}};
\node (b7) at (3.1,-1.6) {\scalebox{0.8}{$\hdots$}};
\node (b8) at (4.4,-1.5) {\scalebox{0.8}{$\Ext^{n-2}_\Upgamma(M,M)$}};
\node (b9c) at (6.5,-1.5) {\scalebox{0.8}{$\Ext^{n-1}_\Upgamma(M,M)$}};
\end{tikzpicture}
\]
\end{prop}
In (almost) all of our applications the map $\Hom_{\Upgamma}(M,M) \to \Ext^{n}_{\Upgamma}(M,M)$ is indeed injective, however we do also compute the cohomology of $\scrT$ in general.  When the map is not injective the answer is just mildly more technical to state; for full details, see \ref{eq.coh-t}.

\medskip
To prove \ref{eq.coh-t_intro} we show that $\scrT$ is quasi-isomorphic to a more homological $\DG$-algebra.  The complex $\PP$ has cohomology in two degrees, but regardless, it is just a bounded complex of projectives, and so we can consider its endomorphism $\DG$-algebra $ \scrEnd_\Upgamma(\PP)$.

\begin{thm}[\ref{thm:qis of scrT}]
There is a natural morphism $ \scrEnd_\Upgamma(\PP)\to\scrT$, which is a quasi-isomorphism of $\DG$-algebras.
\end{thm}
On one hand, this means we can interpret $\scrT$ simply as a more explicit $\DG$ model of the more homological $\scrEnd_\Upgamma(\PP)$.  On the other hand, the naturality of the definition of $\scrT$ gives some motivation to then consider the slightly less motivated algebra $\scrEnd_\Upgamma(\PP)$.

\subsection{The unitally positive algebra}\label{subsec: intro UPsub}
This subsection introduces $\scrN$, the unitally positive $A_\infty$-algebra.  This object can be associated to (almost) any $\DG$-category, but for simplicity this introduction restricts to the case of $\DG$-algebras.   

The motivation is simple: given a $\DG$-algebra $\scrB$, it is well known that in general the good truncation $\uptau_{\geq 0}\scrB$ need not be a $\DG$-algebra.  This is not the only issue.  Regardless of whether it is a $\DG$-algebra or not, in degree zero $\HH^0(\uptau_{\geq 0}\scrB)=\HH^0(\scrB)$.  Often in examples (e.g.\ $\scrB=\scrT$ from the previous section) the cohomology in degree zero is a direct sum, and we only want to `access' part of it.

\medskip
For this reason, we now pass to the $A_\infty$-level, and consider a strictly unital minimal model of $\scrB$.  It is elementary to see, in \ref{lem.connective-truncation}, that this $A_\infty$-structure on the cohomology of $\scrB$ restricts to give an $A_\infty$-structure on $\mathbb{C}\mathsf{1} \oplus \HH^{>0}(\scrB)$, where $\mathsf{1}\colonequals [\Id_\scrB] \in \HH^0(\scrB)$. 

\begin{defin}[\ref{definition nonneg subalg}]\label{definition nonneg subalg intro}
Given a unital $\DG$-algebra $\scrB$ such that the unit is not a coboundary, we call the $A_\infty$-algebra $\scrN_\scrB\colonequals \mathbb{C}\mathsf{1} \oplus \HH^{>0}(\scrB)$, with $A_\infty$-structure induced from a strictly unital minimal model of $\scrB$ (see \ref{lem.connective-truncation}), the \emph{unitally positive $A_\infty$-algebra} associated to $\scrB$.
\end{defin}
Returning to the running setting of the periodic extension $\DG$-algebra $\scrT$, at least in the case when $\Hom_\Upgamma(M,M)\cong \mathbb{C}$, the unitally positive $A_\infty$-algebra $\scrN_\scrT$ has the same cohomology as just the bottom row of \ref{eq.coh-t_intro}.  Thus, \emph{very} roughly speaking, $\scrN_\scrT$ can be viewed as a `suitably positive' truncation of the $\DG$-algebra $\scrT$. Given it also contains the unit, this explains the name.

Of course, we show that in general the unitally positive $A_\infty$-algebra is well defined, up to quasi-isomorphism in the category of unital $\DG$-algebras.

\begin{prop}[\ref{scrN is well-defined}]
If two unital $\DG$-algebras $\scrA$, $\scrB$ are quasi-isomorphic as unital $\DG$-algebras, there exists an $A_\infty$-quasi-isomorphism between $\scrN_\scrA$ and $\scrN_\scrB$.
\end{prop}
In fact, all constructions and results in this subsection extend to (almost) all $\DG$-categories, and this extension is needed to cover the multi-curve version of the application in the next section; the reader is referred to \S\ref{subsec:unitalposCat} for full details.

\subsection{Application: The Donovan--Wemyss conjecture}
Recall that to every  $3$-fold flopping contraction $X\to\Spec R$, one can associate a finite dimensional noncommutative algebra $\Acon$, known as the contraction algebra \cite{DW1, DW3}. In this introduction we will restrict to the single-curve case, where it is well-known that $\Acon$ is local and the unique simple $\Acon$-module $S$ admits a length four complex $
\PP$ of finitely generated projective $\Acon$-modules which induce a periodic projective resolution; the multi-curve case is also solved in this paper, but details are left to \S\ref{subsec:DWmulti}.

Applying the above \S\ref{subsec: intro TEDGA} and \S\ref{subsec: intro UPsub}  to $S\in\mod\Acon$ constructs the unitally positive $A_\infty$-algebra $\scrN_S\colonequals \mathbb{C}\mathsf{1}\oplus \HH^{>0}(\scrEnd_{\Acon}(\PP))$.  In this flops setting we can also view $S$ as a module over the corresponding noncommutative crepant resolution $\rmA$ in \eqref{eqn:NCCR}, where it has a finite projective resolution $\scrQ\to S$.  Our main technical result is the following.

\begin{thm}[\ref{Main DG result}]\label{Main DG result intro}
There is an $A_\infty$-quasi-isomorphism between $\scrN_S$ and 
$\scrEnd_{\rmA}(\scrQ)$.
\end{thm}

Roughly speaking, the point is that the $3$-CY object $\scrEnd_{\rmA}(\scrQ)$ can be reconstructed from the finite dimensional algebra $\Acon$, via considering its simple module $S$ and its associated unitally positive $A_\infty$-algebra $\scrN_S$. The precise statement is \ref{cor:main_reconstruct}.

\medskip
The following is the main consequence, which reduces the classification of $3$-fold flops in birational geometry to that of finite dimensional algebras.

\begin{cor}[\ref{main text local}, Donovan--Wemyss conjecture]\label{main intro}
Suppose that $X_1\to\Spec R_1$ and $X_2\to\Spec R_2$ are two $3$-fold flopping contractions, where both $X_i$ are smooth, both $R_i$ are complete local, and both contractions have precisely one curve above the origin.  Write $\Acon$ and $\Bcon$ for their corresponding contraction algebras.  Then
\[
R_1\cong R_2\iff \Acon\cong\Bcon.
\]
\end{cor}
The multi-curve version of \ref{main intro} is also proved, in \ref{main text}.   In either case, the point is that it is an immediate consequence of \ref{Main DG result intro} and Koszul duality that, in this smooth setting, the contraction algebra $\Acon$ recovers its derived/DG version. This last statement is what originally motivated this paper: the construction of $\scrN_S$ and then \ref{Main DG result intro} should be viewed simply as the very direct way to achieve it. With this result in hand, both \ref{main intro} and its multi-curve analogue \ref{main text} then follow quickly, since by \cite{Booth, HuaKeller}, building on work of Kalck--Yang \cite{KalckYang} and others, it is well-known that the derived contraction algebra determines the isomorphism class of $R$.

%

\subsection*{Relation to Literature}
This paper gives the second proof of \ref{main intro}; the first was \cite{JM}.  The main benefits here are (a) the directness of the proof, and (b) the fact that the constructions made along the way are general. Indeed, perhaps the most remarkable aspect is that aside from both papers proving \ref{main intro}, there is no technical overlap in either the approaches or in the settings considered.  Consequently, there is no overlap in any of the results that lead to~\ref{main intro}. The paper \cite{JM} restricts to a cluster tilting setting, and proves very strong bijection results in that particular setting; here the setting is broadly much more general (e.g.\ $\scrT$ requires only a single periodic module, not a periodic algebra, and the construction of $\scrN$ is very general indeed), and so the results here necessarily have a different flavour, given that bijection results do not hold in greater generality.

That said, it does seem notable that neither of the two approaches to \ref{main intro} can prove the conjecture whilst staying entirely within the world of $\DG$-algebras and categories.  Both need to pass to the $A_\infty$-level in some fundamental way, in order to access structure which is currently not visible within the $\DG$-world.

\medskip
Versions of the unitally positive $A_\infty$-algebra $\scrN$ have appeared implicitly in the literature in various settings, albeit (to the best of our knowledge) always applied to $\DG$-endomorphism rings of objects in particular abelian categories, which automatically have no negative Ext groups. One example can be found in highest weight categories, where the `box algebra' associated to the direct sum of the standard modules can be interpreted as an $\scrN$ as constructed here; see e.g.\ \cite[under 4.35]{Kulshammer} and the references therein to \cite{Ovsienko, KKO}. A similar construction, in deformation theory, appears in the talks  of Bodzenta--Bondal.  Neither applies directly to our setting: the general construction of $\scrN$ in \ref{definition nonneg subalg} below applies to (almost) any $\DG$-algebra, and the novelty in this paper is to use this generality to both remove negative cohomology, and also to tweak the degree zero part.

\subsection*{Conventions}

Throughout, script font means that the object should be considered together with its $\DG$ or $A_\infty$ structure, whichever is clear from the context.  So, $\scrEnd(P)$ and $\scrHom$ are graded over $\mathbb{Z}$. All $\DG$-algebras $\scrA$ will be unital, such that the identity is not a coboundary.   We will denote $A_\infty$-morphisms and functors by $\scrA \rightsquigarrow \scrB$.

We will always draw complexes with arrows pointing to the right, namely
\[
x\colonequals \hdots\to a\to b\to c\to \hdots
\]
Irrespective of how the objects in the complex $x$ are notated or numbered, $\HH^{-1}(x)$ will always be the cohomology immediately to the left of the zero position.  So, if $b$ is in degree zero above, $\HH^{-1}(x)$ is the cohomology at $a$. Write $[1]$ for the shift to the left, with a sign on the differential.  Thus 
\[
x[1]\colonequals \hdots\xrightarrow{-\updelta} b\xrightarrow{-\updelta} c\xrightarrow{-\updelta} d\to \hdots
\]
where $c$ is now in degree zero. Given an ring $\Upgamma$, modules will be right modules. We write $\circ f$ for precomposition with the map $f$, i.e., $g \mapsto g\circ f$, and write $f\circ$ for postcomposition with  $f$, i.e., $g \mapsto f \circ g$.

\subsection*{Acknowledgments} The authors would like to thank Matt Booth, Alberto Canonaco, Will Donovan, Zhuang Hua, Gustavo Jasso, and Bernhard Keller for helpful conversations. Thanks also to Julian K\"ulshammer for pointing out the connections to box algebras, and to Fernando Muro for explaining Remark~\ref{rem.adjoint-to-incl}.  Special thanks are due to Martin Kalck for patiently explaining many aspects of \cite{KalckYang} to the third author in 2012.  Many years later, this formed the strategy to classification, via relative singularity categories.

\subsection*{Funding} J.K.\ was supported by EPSRC grant EP/M017516/2, both E.L.\ and M.W.\ by ERC Grant 101001227 (MMiMMa), and M.W.\ additionally by EP/R034826/1. E.L.\ is also a member of the GNSAGA group of INDAM.

\subsection*{Open Access} For the purpose of open access, the authors have applied a Creative Commons Attribution (CC:BY) licence to any Author Accepted Manuscript version arising from this submission.

\section{Trivial extension DG-algebras}

Given any algebra $\Upgamma$, and any $\Upgamma$-module that admits a periodic projective resolution (in a rather weak sense), this section constructs an associated \emph{trivial extension $\DG$-algebra}.  In particular, this setup applies to $\Upgamma$ itself.

\subsection{Generalities}
Throughout, let $\Upgamma$ be a $k$-algebra, where $k$ is some field.

\begin{setup}\label{key setup}
Suppose that $M\in\mod \Upgamma$ admits a length $n\geq 1$ periodic projective resolution, in the sense that there exists a complex of finitely generated projective $\Upgamma$-modules
\[
\PP \colonequals \quad0 \to P_{n-1} \xrightarrow{d_{n-1}} P_{n-2} \xrightarrow{d_{n-2}} \dots \xrightarrow{d_2} P_{1} \xrightarrow{d_{1}} P_0 \to 0 
\]
such that 
\[
0 \to M \xrightarrow{\upbeta}  P_{n-1} \to P_{n-2} \to \dots \to P_{1} \to P_0 \xrightarrow{\upalpha} M  \to 0
\]
is exact.  Mildly abusing notation, by shifting the complex $\PP $, and using the connecting map $d_0\colonequals \upbeta\circ\upalpha$, we construct the complex
\[
\scrP\colonequals\quad \dots\to  \PP [2n] \xrightarrow{d_0[n]} \PP [n] \xrightarrow{d_0} \PP  \to 0
\]
where each component $\PP [\ell n]$ represents a length $n$ complex of projective $\Upgamma$-modules.  For calibration, the complex $\scrP$ is written explicitly in the cases $n=1,2,3$ in \ref{ex:spelloutscrP} below.

From this, consider the $\DG$-algebra 
\[
\scrE\colonequals \scrEnd_{\Upgamma}(\scrP)
\]
in its usual graded pieces $\scrE=\bigoplus_{i\in\mathbb{Z}}\, \scrE^i$, where for calibration $\scrE^1$ is the collection of all maps that point one position to the right
\[
\begin{tikzpicture}
\node (a1) at (0,0) {$0$};
\node (a2) at (-1,0) {$P_0$};
\node (a3) at (-2,0) {$P_1$};
\node (a4) at (-3,0) {$P_2$};
\node (a5) at (-4,0) {$P_3$};
\node (a6) at (-5,0) {$\hdots$};
\node (a7) at (-6.1,0) {$P_0$};
\node (a8) at (-7.1,0) {$P_1$};
\node (a9) at (-8,0) {$\hdots$};
\node (b1) at (0,-1) {$0$};
\node (b2) at (-1,-1) {$P_0$};
\node (b3) at (-2,-1) {$P_1$};
\node (b4) at (-3,-1) {$P_2$};
\node (b5) at (-3.9,-1) {$\hdots$};
\node (b6) at (-5,-1) {$P_{n-1}$};
\node (b7) at (-6.2,-1) {$P_0$};
\node (b8) at (-7.2,-1) {$P_1$};
\node (b9) at (-8,-1) {$\hdots$};
\draw[->] (a2)--(b1);
\draw[->] (a3)--(b2);
\draw[->] (a4)--(b3);
\draw[->] (a5)--(b4);
\draw[->] (a7)--(b6);
\draw[->] (a8)--(b7);
\draw[->] (a2)--(a1);
\draw[->] (a3)--(a2);
\draw[->] (a4)--(a3);
\draw[->] (a5)--(a4);
\draw[->] (a6)--(a5);
\draw[->] (a7)--(a6);
\draw[->] (a8)--(a7);
\draw[->] (b2)--(b1);
\draw[->] (b3)--(b2);
\draw[->] (b4)--(b3);
\draw[->] (b5)--(b4);
\draw[->] (b6)--(b5);
\draw[->] (b7)--(b6);
\draw[->] (b8)--(b7);

\end{tikzpicture}
\]
and there are no commutativity (or other) conditions on this collection.  We will write $\scrE=\bigoplus_{i\in\mathbb{Z}}\, \scrE^i$ with grading $|\,|$, multiplication $\circ$, and differential $\updelta$ (see e.g.\ \cite[\S1.1]{DerivingDG}).
\end{setup}

The $\DG$-algebra $\scrE$ depends on both $\Upgamma$ and $M$, but this is suppressed from the notation.  We also remark that it is well-known that different choices of projective resolutions produce quasi-isomorphic $\DG$-algebras $\scrE$.

Implicit in the above setup is the convention that shifting the complex $\PP $ by an odd homological degree introduces a minus sign on the differentials.  Thus, when $n$ is odd, some minus signs occur.   

\begin{example}\label{ex:spelloutscrP}
For $n=1$ the complex $\scrP$ is 
\[
\dots\xrightarrow{\upbeta\circ\upalpha} P_0  \xrightarrow{-\upbeta\circ\upalpha} P_{0} \xrightarrow{\upbeta\circ\upalpha} P_0\to 0.
\]
For $n=2$ the complex $\scrP$ is
\[
\dots\to P_1\xrightarrow{d_1} P_{0} \xrightarrow{\upbeta\circ\upalpha} P_{1} \xrightarrow{d_1} P_0  \xrightarrow{\upbeta\circ\upalpha} P_{1} \xrightarrow{d_1} P_0\to 0 ,
\]
and for $n=3$ it is
\[
\dots\xrightarrow{d_2} P_1\xrightarrow{d_1} P_{0} \xrightarrow{-\upbeta\circ\upalpha} P_{2} \xrightarrow{-d_2} P_{1} \xrightarrow{-d_1} P_0  \xrightarrow{\upbeta\circ\upalpha} P_{2} \xrightarrow{d_2} P_{1} \xrightarrow{d_1} P_0\to 0 .
\]
\end{example}

\begin{example} \label{Example: Acon periodic resolution}
One module that always has such a periodic resolution (for any $n\geq 2$) is the projective module $\Upgamma$ itself: we can take the complex
\[
\PP \colonequals \quad0\to \Upgamma \to 0 \to\dots\to 0\to\Upgamma\to 0
\]
with $\upalpha=\upbeta=\Id$. 
\end{example} 

\begin{notation}
Under Setup~\ref{key setup}, consider the two elements $\upsigma, \uptau \in \scrE$ defined as
\[
\upsigma\colonequals
\begin{tikzcd}
\phantom{\PP [2n]}  \arrow{r}\arrow{rd}{\Id} & \PP [2n] \arrow{r} \arrow{rd}{\Id} & \PP [n] \arrow{r} \arrow{rd}{\Id} & \PP  \arrow{r} \arrow{rd}{0} & 0 \arrow{r} \arrow{rd}{0} & \dots\\
\phantom{\PP [2n]}  \arrow{r}  & \PP [2n] \arrow{r}  & \PP [n] \arrow{r} & \PP \arrow{r}  & 0 \arrow{r} & \dots
\end{tikzcd}
\]
and 
\[
\uptau\colonequals
\begin{tikzcd}
\phantom{\PP [2n]} \arrow{r} & \PP [2n] \arrow{r} \arrow{ld}{\Id} & \PP [n] \arrow{r} \arrow{ld}{\Id} & \PP  \arrow{r} \arrow{ld}{\Id} & 0 \arrow{r} \arrow{ld}{0} & \dots \arrow{ld}{0}\\
\phantom{\PP [2n]}  \arrow{r}   & \PP [2n] \arrow{r}  & \PP [n] \arrow{r} & \PP  \arrow{r} & 0 \arrow{r} & \dots
\end{tikzcd}
\]
\end{notation}

Note that $|\upsigma\kern 1pt|=n$ and $|\uptau\kern 1pt|=-n$. It is a straightforward calculation to see that
\begin{equation}
\upsigma \circ \uptau = \Id \quad \text{ and} \quad \uptau \circ \upsigma = \Id_{\ge n}\label{sig tau basic}
\end{equation}
where $\Id=\Id_{\scrE}$ and $\Id_{\ge n}$ denotes the degree 0 map obtained as the truncation of $\Id$ that restricts to the identity on the components $\PP [\ell n]$ for all $\ell\ge 1$ and the zero map on the component $\PP $. 

\begin{lemma}\label{Lemma: deltasigma is zero}
Under Setup~\textnormal{\ref{key setup}}, $\updelta (\upsigma) = 0$.
\end{lemma}
\begin{proof}
This is an easy verification.  In components, either $\upsigma=\Id\colon P_i[\ell n] \rightarrow P_{i}[(\ell-1) n]$ with $\ell>0$, or $\upsigma=0$, so we just need to check that $\updelta(\upsigma)=0$ on the components $P_i[\ell n]$ of $\scrP$ with $\ell>0$.  When $i \neq 0$, for any $\ell>0$ we have
\begin{align*}
\updelta(\upsigma) & = d_i[(\ell-1)n] \circ \Id - (-1)^{|\upsigma|} \Id \circ\, d_i[\ell n] \\
& = (-1)^{(\ell -1)n}d_i - (-1)^{n+\ell n} d_i \tag{$|\upsigma|=n$}\\
& = (-1)^{\ell n +n}d_i - (-1)^{n+\ell n}d_i\tag{$(-1)^n=(-1)^{-n}$}\\
& =0.
\end{align*}
Similarly, for $i=0$ with $\ell>1$ then
\begin{align*}
\updelta(\upsigma) & = (\upbeta\circ\upalpha)[(\ell-1)n] \circ \Id - (-1)^{|\upsigma|} \Id \circ (\upbeta\circ\upalpha)[\ell n]  \\
& = (-1)^{(\ell -1)n}\upbeta\circ\upalpha -(-1)^{n+\ell n} \upbeta\circ\upalpha  \\
&=0,
\end{align*}
and for $i=0$ with $\ell=1$, we have $\updelta(\upsigma) = 0 \circ \Id - (-1)^{|\upsigma|}\, 0 \circ (\upalpha \circ \upbeta)[n]=0$.
\end{proof}

\subsection{The trivial extension $\DG$-algebra}
 
 The construction of the trivial extension $\DG$-algebra involves   isolating two periodic pieces of $\scrE$, the first of which is the following.

\begin{defin}
Under Setup~\ref{key setup}, write $\sigE\colonequals \bigoplus_{i\in\mathbb{Z}}\sigE^i$, where
\[
\sigE^i \colonequals \{ x \in \scrE^i \mid \upsigma \circ x = (-1)^{|\upsigma||x|}x \circ \upsigma \}.
\]
\end{defin}

\begin{lemma} \label{Lemma: sigE is DG} 
With notation as above, the following hold.
\begin{enumerate}
\item\label{Lemma: sigE is DG 1}   If $x\in\sigE$ is homogeneous, then $\upsigma\circ\updelta(x)=(-1)^{|\upsigma|(|x|+1)}\updelta(x)\circ\upsigma$.
\item\label{Lemma: sigE is DG 2}   $\sigE$ is a $\DG$-subalgebra of $\scrE$.
\end{enumerate}
\end{lemma} 
\begin{proof}
(1) On one hand, by the graded Leibniz rule  
\[
\updelta(\upsigma \circ x)=\updelta(\upsigma) \circ x +(-1)^{|\upsigma|}\upsigma \circ \updelta(x)\stackrel{\scriptstyle\ref{Lemma: deltasigma is zero}}{=}(-1)^{|\upsigma|}\upsigma \circ \updelta(x).
\]
On the other hand,   
\begin{align*}
\updelta(\upsigma \circ x)&=(-1)^{|\upsigma||x|}\updelta(x\circ\upsigma)\tag{since $x\in\sigE$}\\
&=(-1)^{|\upsigma||x|}\big(\updelta(x)\circ \upsigma +(-1)^{|x|}x\circ \updelta(\upsigma)\big)\tag{graded Leibniz rule}\\
&=(-1)^{|\upsigma||x|}\updelta(x)\circ \upsigma\tag{$\updelta(\upsigma)=0$ by \ref{Lemma: deltasigma is zero}}.
\end{align*}
Equating terms and multiplying by $(-1)^{|\upsigma|}$ gives $\upsigma\circ\updelta(x)=(-1)^{|\upsigma|(|x|+1)}\updelta(x)\circ\upsigma$.\\
(2) Clearly $\Id\in\sigE$, and further $\sigE$ is a graded subspace of $\scrE$.  It is closed under multiplication since for homogeneous $x,y\in {^{\upsigma}\scrE}$ 
\[
\upsigma \circ (x \circ y) 
=
(-1)^{|\upsigma||x|}x \circ \upsigma \circ y = (-1)^{|\upsigma|(|x|+|y|)}(x \circ y) \circ \upsigma.
\]
The differential $\updelta$ on $\scrE$ restricts to a differential on $\sigE$ since for homogeneous $x\in \sigE$, by part (1) $\upsigma\circ\updelta(x)=(-1)^{|\upsigma|(|x|+1)}\updelta(x)\circ\upsigma=(-1)^{|\upsigma||\updelta(x)|}\updelta(x)\circ\upsigma$, and hence $\updelta(x)\in\sigE$.
\end{proof}

The second periodic piece of the trivial extension algebra is the following, which is very similar to the above, up to a sign.
\begin{defin}\label{def nSig}
Under Setup~\ref{key setup}, write $\nsigE\colonequals\bigoplus_{i\in\mathbb{Z}}\nsigE^i$, where
\[
\nsigE^i \colonequals  \{ y \in \scrE^i \mid \upsigma \circ y = (-1)^{|\upsigma|(|y|+1)}y \circ \upsigma \}
\]
\end{defin}

The following is the analogous version of \ref{Lemma: sigE is DG}.  Recall that a $\DG$  $\sigE$-module is a $\mathbb{Z}$-graded $\sigE$-module endowed with a compatible differential \cite[p1]{DerivingDG}.
\begin{lemma} \label{Lemma: nsigE is DG}
With notation as above, the following hold.
\begin{enumerate}
\item\label{Lemma: nsigE is DG 1} If $y\in\nsigE$ is homogeneous, then $\upsigma\circ\updelta(y)=(-1)^{|\upsigma||y|}\updelta(y)\circ\upsigma$.
\item\label{Lemma: nsigE is DG 2} $\nsigE$ is a $\DG$ $\sigE$-bimodule.
\end{enumerate}
\end{lemma} 
\begin{proof}
(1) The proof follows in a very similar manner to \ref{Lemma: sigE is DG}\eqref{Lemma: sigE is DG 1}.\\
(2) We first claim that $\nsigE$ is a $\mathbb{Z}$-graded $\sigE$-module, on both the left and right.  This holds since it is clear that $\nsigE$ is a graded vector subspace of $\scrE$, and further $\nsigE$ is closed under left and right multiplication by $\sigE$, since 
for homogeneous $x \in \sigE$ and $y \in \nsigE$
\begin{align*}
\upsigma \circ x \circ y &= (-1)^{|\upsigma|(|x|+|y|+1)} x \circ y \circ \upsigma,\\
\upsigma \circ y \circ x &= (-1)^{|\upsigma|(|x|+|y|+1)} y \circ x \circ \upsigma.
\end{align*}
This makes $\nsigE$ a $\mathbb{Z}$-graded right and left $\sigE$-module. The bimodule compatibility condition for the left and right action, $a \circ ( m \circ b) = (a \circ m ) \circ b$ for $a,b \in \sigE$ and $m \in \nsigE$, holds since the actions are those inherited from the associative multiplication in $\scrE$. 

Now the differential on $\scrE$ restricts to a differential on $\nsigE$ since for homogeneous $y\in \nsigE$, by (1), $\upsigma\circ\updelta(y)=(-1)^{|\upsigma||y|}\updelta(y)\circ\upsigma=(-1)^{|\upsigma|(|\updelta(y)|+1)}\updelta(y)\circ\upsigma$, and so $\updelta(y)\in\nsigE$.   This differential is compatible with both the left and right action, namely
\[
\updelta(a \circ m \circ b )= \updelta(a) \circ m \circ b + (-1)^{|a|} a \circ \updelta(m) \circ b + (-1)^{|a|+|m|} a \circ m \circ \updelta(b)
\]
holds for $a \in \sigE^i$, $m \in \nsigE^j$, and $b \in \sigE$, since all differentials are inherited from $\scrE$, which obeys the Leibniz rule.  
\end{proof}

\begin{remark} 
If $|\upsigma|$ is even, $\sigE=\nsigE$. In this case, $\nsigE$ is clearly a $\DG$ $\sigE$-bimodule.
\end{remark}

The following is our main new object.
\begin{defin}\label{Definition: TEA}
Under the Setup~\ref{key setup}, the trivial extension $\DG$-algebra $\scrT$ of $M\in\mod\Upgamma$ is defined as follows.  As a graded vector space,
\[
\scrT \colonequals \sigE \oplus \nsigE[1-n].
\]
To avoid confusion on signs associated with the shift, we will always view a homogenous element of $\scrT$ of degree $j$ to be a pair $(x,y)$ where $x \in \sigE^j$ and $y \in \nsigE^{j-n+1}$ are homogeneous elements of degrees $|x|=j$ and $|y|=j+1-n$ respectively; equivalently $j=|x|=|y|+n-1$.  Thus, in this convention both $x$ and $y$ will be viewed in the fixed $\scrE$, and so the fixed differential $\updelta_\scrE=\updelta$ can be applied to both.

Multiplication is defined on homogeneous elements by
\[
(x,y) \cdot (a,b) = \big( x \circ a, \, x \circ b + (-1)^{|a|(|\upsigma|+1)} y \circ a \big),
\]
and extended linearly. The differential $\upxi$ on $\scrT$ is defined on homogeneous elements by 
\[
\upxi (x,y) \colonequals \big(\updelta(x) - (-1)^{|x|} y \circ \upsigma, \,\updelta(y)\big)
\]
and extended linearly.
\end{defin}
A priori, $\scrT$ depends on $\PP$, $\upalpha$ and $\upbeta$, but this is suppressed from the notation.  As the above multiplication is a graded version of the multiplication of the usual trivial extension algebra, this partially justifies the name.

\begin{remark}\label{rem:actually_a_cone}
The differential $\upxi$ is just the differential of the cone of the degree $n$ morphism $\nsigE \to \sigE$ sending $y$ to $(-1)^{|\upsigma|} y \circ \upsigma$ and as such, it is very natural to consider. Here, instead of the usual sign convention on the differential of the cone of $f \colon V \to W$, which is $d_C(x,y)= (d_Wx + f(y), -d_V y)$, for convenience and comparison with $\upxi$ we are using the alternative convention that $d'_C(x,y)=(d_Wx + (-1)^{|y|}f(y), d_V y)$.
\end{remark}

\begin{prop}\label{Triv DG is DG}
Under the Setup~\textnormal{\ref{key setup}}, $(\scrT,\cdot,\upxi)$ is a $\DG$-algebra.
\end{prop}
\begin{proof}
We first check that multiplication is well-defined, namely that the two components remain in $\sigE\oplus\nsigE[1-n]$ under multiplication. This is clear in the first component, and in the second component this follows since
\begin{align*}
\hspace{3em}&\hspace{-3em} \upsigma \circ (x \circ b + (-1)^{|a|(|\upsigma|+1)} y \circ a)\\
& = \upsigma \circ x \circ b + (-1)^{|a|(|\upsigma|+1)} \upsigma \circ  y \circ a \\ 
& = (-1)^{|\upsigma||x|} x \circ \upsigma \circ b + (-1)^{|a|(|\upsigma|+1)+ |\upsigma|(|y|+1)}  y \circ \upsigma \circ  a \\ 
&=  (-1)^{|\upsigma|(|x|+|b|+1)} x \circ b \circ \upsigma  + (-1)^{|a|(|\upsigma|+1)+|\upsigma|(|y|+|a|+1)} y \circ a  \circ \upsigma \\ 
&=  (-1)^{|\upsigma|(|x|+|b|+1)}(x \circ b  + (-1)^{|a|(|\upsigma|+1)} y \circ a ) \circ \upsigma
\end{align*}
as $|x|+|b|=|y|+|a|$.

Now $\scrT$ is clearly a $\mathbb{Z}$-graded vector space, with a  multiplication that is compatible with the grading.  To check that $\scrT$ forms an algebra, we note that $(\Id,0)\in\scrT$ acts as the identity element, and associativity follows from the calculation
\begin{align*}
\hspace{3em}&\hspace{-3em}\big((a,b)\cdot (c,d)\big) \cdot (e,f)\\
 & =\big(a \circ c,a \circ d +(-1)^{(|\upsigma|+1)|c|}b \circ c\big) \cdot (e,f) \\
 & =(a\circ c \circ e,a \circ c \circ f + (-1)^{(|\upsigma|+1)|e|}a \circ d \circ e +(-1)^{(|\upsigma|+1)(|c|+|e|)}b \circ c \circ e) \\
 &= (a,b) \cdot \big(c \circ e, c \circ f + (-1)^{(|\upsigma|+1)|e|}  d \circ  e \big) \\
 &= (a,b) \cdot \big((c,d) \cdot (e,f)\big).
\end{align*}
Multiplication distributes over addition by definition, and thus $\scrT$ is an algebra.

We next claim that $\upxi$ is a differential.  First, we claim that it is well defined on the level of sets. For homogeneous $(x,y)\in\scrT$, we know from \ref{Lemma: sigE is DG} that $\updelta(x)\in\sigE$, and from \ref{Lemma: nsigE is DG} that $\updelta(y)\in\nsigE$.  Also, $|y\circ\upsigma|=|y|+n=|x|+1=|\updelta(x)|$ with
\begin{align*}
\upsigma \circ ( y \circ \upsigma) & = \upsigma \circ y \circ \upsigma \\
& = (-1)^{ |\upsigma|(|y|+1)} y \circ \upsigma \circ \upsigma \tag{$y\in\nsigE$} \\
& = (-1)^{ |\upsigma|(|x| + |\upsigma|)} (y \circ \upsigma )\circ \upsigma  \tag{$|y|+1 \equiv_2 |x|+|\upsigma| $}\\
& = (-1)^{ |\upsigma|(|x| + 1)} (y \circ \upsigma )\circ \upsigma. \tag{$|\upsigma|^2 \equiv_2 |\upsigma| $}\\
& = (-1)^{ |\updelta(x)||\upsigma|} (y \circ \upsigma )\circ \upsigma. 
\end{align*}
Combining, $\upxi (x,y) \colonequals \big(\updelta(x) - (-1)^{|x|} y \circ \upsigma, \,\updelta(y)\big)\in\sigE\oplus\nsigE$ with the correct degree in the two factors, and so $\upxi$ is well defined.  It is clear that $\upxi$ has degree one on $\scrT$.

It is easy see that $\upxi^2=0$, by just observing
\begin{align*}
\upxi^2(x,y) & = \upxi \left(\updelta(x) - (-1)^{|x|} y \circ \upsigma, \,\updelta(y) \right) \\
&= \left(\updelta^2(x) -(-1)^{|x|}\updelta(y) \circ \upsigma - (-1)^{|\updelta(x)|}\updelta(y) \circ \upsigma, \,\updelta^2(y) \right)\\ 
&=(0,0).
\end{align*}
where we have used the graded Leibniz rule, and the fact that $\updelta(\upsigma)=0$ by \ref{Lemma: deltasigma is zero}. Thus $\upxi$ is a well-defined differential.

To establish that $\scrT$ is a $\DG$-algebra, it remains to check that the graded Leibniz rule
\[
\upxi \big((x,y)\cdot(a,b) \big) = \upxi(x,y) \cdot (a,b) + (-1)^{|x|}(x,y) \cdot \upxi (a,b)
\]
is satisfied, where $(x,y)$ and $(a,b)$ are assumed to be homogeneous elements, so $|x|=|y|+n-1$ and similarly $|a|=|b|+n-1$.  On one hand,
\begin{multline*}
\upxi\big((x,y)\cdot(a,b)\big) = \upxi \big(x \circ a, x \circ b + (-1)^{|a|(|\upsigma|+1)} y \circ a \big) = \\
 \left(
\begin{array}{c} 
\updelta(x) \circ a + (-1)^{|x|}x \circ \updelta(a) - (-1)^{|x|+|a|}x \circ b \circ \upsigma  - (-1)^{|a|(|\upsigma|+1) + |a|+ |x|}y \circ a \circ \upsigma ,\\
 \updelta(x) \circ b + (-1)^{|x|} x \circ \updelta(b) + (-1)^{|a|(|\upsigma|+1)} \updelta(y) \circ a + (-1)^{|a|(|\upsigma|+1)+|y|} y \circ \updelta(a) \kern 14pt
\end{array}
\right).
\end{multline*}
On the other hand,
\begin{multline*}
 \upxi(x,y) \cdot (a,b) + (-1)^{|x|}(x,y) \cdot \upxi (a,b) = \\
 \left(\updelta(x) - (-1)^{|x|}y \circ \upsigma, \updelta(y) \right) \cdot (a,b) + (-1)^{|x|}(x,y) \cdot \left( \updelta(a) - (-1)^{|a|}b \circ \upsigma, \updelta(b) \right).
\end{multline*}
Displaying the first and second component separately, we see that the first component is
\[
\updelta(x) \circ a - (-1)^{|x|} y \circ \upsigma \circ  a + (-1)^{|x|} x \circ \updelta(a) - (-1)^{|x|+|a|}x \circ b \circ \upsigma
\]
and the second component is
\begin{multline*}
\updelta(x) \circ b - (-1)^{|x|} y \circ \upsigma \circ b + (-1)^{|a|(|\upsigma|+1)}\updelta(y) \circ a 
+ (-1)^{|x|} x \circ \updelta(b)  \\+ (-1)^{|x|+|\updelta(a)|(|\upsigma|+1)}y \circ \updelta(a) - (-1)^{|x|+|a|+|\updelta(a)|(|\upsigma|+1)} y \circ b \circ \upsigma.
\end{multline*}
As such, to satisfy the graded Leibniz rule the required equality in the first component reduces to the identity
\[
(-1)^{|a|(|\upsigma|+2)}y \circ a \circ \upsigma =  (-1)^{|a||\upsigma|}y \circ a \circ \upsigma = y \circ \upsigma \circ  a,
\]
which is true since $a \in \sigE$.  The required equality in the second component reduces to 
\begin{multline*}
(-1)^{|a|(|\upsigma|+1)+|y|} y \circ \updelta(a) =- (-1)^{|x|} y \circ \upsigma \circ b \\+ (-1)^{|x|+|\updelta(a)|(|\upsigma|+1)}y \circ \updelta(a) - (-1)^{|x|+|a|+|\updelta(a)|(|\upsigma|+1)} y \circ b \circ \upsigma.
\end{multline*}
which is true by the following two calculations. First,
\begin{align*}
(-1)^{|a|(|\upsigma|+1) + |y|} y \circ \updelta(a)
&= (-1)^{|a|(|\upsigma|+1) + |x|- |\upsigma| +1 } y \circ \updelta(a) \tag{$|y|=|x|-|\upsigma| +1$}\\
&= (-1)^{(|a|+1)(|\upsigma|+1) + |x| } y \circ \updelta(a) \tag{$|\upsigma| \equiv_2 - |\upsigma|$} \\
& = (-1)^{|x|+|\updelta(a)|(|\upsigma|+1)} y \circ \updelta(a), \tag{$|\updelta(a)|=|a|+1$}
\end{align*}
and second
\begin{align*}
y \circ \upsigma \circ b & = (-1)^{|\upsigma|(|b|+1)} y \circ b \circ \upsigma  \tag{$\upsigma \circ b = (-1)^{|\upsigma|(|b|+1)}b \circ \upsigma$}\\
& = (-1)^{|\upsigma|(|a|+ |\upsigma|)} y \circ b \circ \upsigma \tag{$|b|+1 \equiv_2 |a| + |\upsigma| $} \\
& = (-1)^{|\upsigma|(|a|+ 1)} y \circ b \circ \upsigma   \tag{$|\upsigma|^2 \equiv_2 |\upsigma| $} \\
& = (-1)^{(|\upsigma|+1)(|a|+ 1) + |a|+1} y \circ b \circ \upsigma \\
&= -(-1)^{|a|+|\updelta(a)|(|\upsigma|+1)} y \circ b \circ \upsigma. \tag{$|\updelta(a)|=|a|+1$}
\end{align*}
Hence the graded Leibniz rule is satisfied, and $\scrT$ is a $\DG$-algebra.
\end{proof}

\section{Quasi-isomorphisms and cohomology}

Under the running Setup~\ref{key setup}, in \S\ref{subsec:qisTtoEnd} we show that the trivial extension $\DG$-algebra $\scrT$ is quasi-isomorphic, as unital $\DG$-algebras, to $\scrEnd_\Upgamma(\PP)$.  This then allows us in \S\ref{sec.cohomology_end} to compute the cohomology of $\scrT$, which it turns out has a very elegant description.

\subsection{Notation and preliminaries}\label{general qis section}
With notation as in \ref{key setup}, consider the inclusion $\upiota\colon \PP\to\scrP$, namely
\[
\begin{tikzpicture}[xscale=1,yscale=1.5]
\node (A) at (2,0) {$\PP$};
	\node (A2) at (4.5,0) {$\hdots$};
	\node (A3) at (6,0) {$0$};
	\node (A4) at (7.5,0) {$0$};
	\node (A5) at (9,0) {$P_{n-1}$};
	\node (A6) at (10.5,0) {$\hdots$};
	\node (A7) at (12,0) {$P_0$};
	\node (A8) at (13,0) {$0$};
	\draw[->] (A2)-- (A3);
	\draw[->] (A3)-- (A4);
	\draw[->] (A4)--(A5);
	\draw[->] (A5)--node[above]{$\scriptstyle d_{n-1}$}(A6);
	\draw[->] (A6)--node[above]{$\scriptstyle d_{1}$}(A7);
	\draw[->] (A7)--(A8);
\node (B) at (2,-1) {$\scrP$};
	\node (B2) at (4.5,-1) {$\hdots$};
	\node (B3) at (6,-1) {$P_1$};
	\node (B4) at (7.5,-1) {$P_{0}$};
	\node (B5) at (9,-1) {$P_{n-1}$};
	\node (B6) at (10.5,-1) {$\hdots$};
	\node (B7) at (12,-1) {$P_0$};
	\node (B8) at (13,-1) {$0$};
	\draw[->] (B2)--node[above]{$\scriptstyle (-1)^nd_2$} (B3);
	\draw[->] (B3)--node[above]{$\scriptstyle (-1)^nd_1$} (B4);
	\draw[->] (B4)--node[above]{$\scriptstyle \upbeta\circ\upalpha$}(B5);
	\draw[->] (B5)--node[above]{$\scriptstyle d_{n-1}$}(B6);
	\draw[->] (B6)--node[above]{$\scriptstyle d_{1}$}(B7);
	\draw[->] (B7)--(B8);
	\draw[double equal sign distance] (A7)--(B7);
	\draw[double equal sign distance] (A5)--(B5);
	\draw[->] (A4)--node[right]{$\scriptstyle $}(B4);
	\draw[->] (A3)--node[right]{$\scriptstyle $}(B3);
\draw[->] (A) --node[right]{$\scriptstyle \upiota$}(B);
\end{tikzpicture}
\]
where the rightmost $P_0$ of the complex is in degree zero.

It is clear from the above diagram that $d_{\scrP} \circ \upiota = \upiota \circ d_{\PP}$, so that $\upiota $ induces an injective morphism of complexes of vector spaces
\[\upiota \circ  \colon \scrEnd_\Upgamma(\PP) \to  \scrHom_{\Upgamma}(\PP , \scrP),\]
so that 
\begin{equation}\label{eq.removes U}
	\upiota \circ  \colon \scrEnd_\Upgamma(\PP) \xrightarrow{\sim} \{ \upiota\circ g\mid g\in\scrEnd_{\Upgamma}(\PP)\}\subseteq\scrHom_{\Upgamma}(\PP,\scrP)\end{equation}
is an isomorphism of complexes of vector spaces.

\begin{notation}\label{notation for DG morph}
	Consider the Setup~\ref{key setup}, and recall that $\scrE=\scrEnd_{\Upgamma}(\scrP)$.
	\begin{enumerate}
		\item\label{notation for DG morph 1} Given $x \in \scrE$, write $x_0 \in \scrHom_{\Upgamma}(\PP , \scrP)$ for the map obtained as the restriction of $x\colon \scrP \rightarrow \scrP$ to the first component $\PP  \rightarrow \scrP$.
		\item\label{notation for DG morph 2} For $f \in \scrHom_{\Upgamma}(\PP , \scrP)$, write $\widebar{f}$ for the element of $\scrE$ obtained by graded repetition.  More precisely,
		\[
		\widebar{f}\colonequals  \left( (-1)^{|\ell n||f|}f_\ell \right)_{\ell=0}^\infty
		\]
		where $f_\ell\colon \colon \PP [\ell n] \rightarrow \scrP$  is just $f$ viewed out of $\PP[\ell n]$ instead of out of $\PP$, which we can do since the codomain $\scrP$ is periodic. If $|f|>0$ then some of the components of $f$ are maps $P_i\to 0$.   These get repeated as the zero map, although their codomain under repetition is no longer $0$.
		Visually, 
		\[
		\widebar{f}=
		\begin{array}{c}
			\begin{tikzpicture}[xscale=0.75,yscale=1.25]
				\node (A0) at (-4,0) {$\hdots$};
				\node (A1) at (-1.75,0) {$\mathrm{P}[3n]$};
				\node (A2) at (0.75,0) {$\mathrm{P}[2n]$};
				\node (A3) at (3,0) {$\mathrm{P}[n]$};
				\node (A4) at (5,0) {$\mathrm{P}$};
				\node (A5) at (7,0) {$0$};
				\node (A6) at (9,0) {$\hdots$};
				\node (a0) at (-4,0) {$\hdots$};
				\draw[->] (A0)--(A1);
				\draw[->] (A1)--(A2);
				\draw[->] (A2)--(A3);
				\draw[->] (A3)--(A4);
				\draw[->] (A4)--(A5);
				\draw[->] (A5)--(A6);
				\node (a1) at (-1.75,-0.1) {};
				\node (a2) at (0.75,-0.1) {};
				\node (a3) at (3,-0.1) {};
				\node (a4) at (5,-0.1) {};
				\node (a5) at (7,-0.1) {};
				\node (a6) at (9,-0.1) {};
				\node (b0) at (-4.5,-1) {};
				\node (b1) at (-2.25,-1) {};
				\node (b2) at (0.25,-1) {};
				\node (b3) at (2.5,-1) {};
				\node (b4) at (4.5,-1) {};
				\node (b5) at (6.5,-1) {};
				\node (b6) at (8.5,-1) {};
				
				\node (B1) at (-4,-1.1) {};
				\node (B2) at (3,-1.1) {$\scrP$};
				\node (B3) at (9,-1.1) {};
				\draw (B1)--(B2);
				\draw (B2)--(B3);
				\draw[->] (a1) --node[gap] {$\scriptstyle (-1)^{|\upsigma||f|}f$} (b1);
				\draw[->] (a2) -- node[gap] {$\scriptstyle f$}(b2);
				\draw[->] (a3) -- node[gap] {$\scriptstyle (-1)^{|\upsigma||f|}f$}(b3);
				\draw[->] (a4) -- node[gap] {$\scriptstyle f$} (b4);
				\draw[->] (a5) -- (b5);
			\end{tikzpicture}
		\end{array}
		\]
		since $|\upsigma|=n$. For $f \in \scrEnd_\Upgamma(\PP)$, likewise write $\widebar{f}$ for the element of $\scrE$ obtained by graded repetition of $\upiota \circ f \in \scrHom_{\Upgamma}(\PP , \scrP)$.
		\item\label{notation for DG morph 3} For $f \in \scrHom_{\Upgamma}(\PP , \scrP)$ or in $ \scrEnd_\Upgamma(\PP)$, set $\Updelta(f)\colonequals\updelta (\widebar{f}\,) - \widebar{\updelta(f)}$.
	\end{enumerate}
\end{notation}

The following is elementary, but important.
\begin{lemma}\label{def U prop}
With notation as above, if $x\in\sigE$, then $x_0$ belongs to $\{ \upiota\circ g\mid g\in\scrEnd_{\Upgamma}(\PP)\}\subseteq\scrHom_{\Upgamma}(\PP,\scrP)$. By the isomorphism \eqref{eq.removes U}, we can then think of $x_0$ as belonging to $\scrEnd_{\Upgamma}(\PP)$.
\end{lemma}
\begin{proof}
Consider the equation $\upsigma \circ x = (-1)^{|\upsigma||x|}x \circ \upsigma$ restricted to $\PP$.  Since $\upsigma$ maps $\PP$ to zero,  $x \circ \upsigma$ is zero on $\PP$, and  it follows that $\upsigma \circ x$ is zero on $\PP$, so $\upsigma\circ x_0=0$.  In particular, any of the non-zero maps which constitute $x_0$ cannot land in $\PP[\ell n]$ with $\ell>0$, since then $\upsigma \circ x_0\neq 0$.  The non-zero maps cannot land to the right of $\PP$, since that is zero.  Hence $x_0$ is a collection of maps, and the non-zero ones necessarily take $\PP$ to $\PP$.  The claim follows.  
\end{proof}

In the remainder of this subsection, working under Setup~\ref{key setup}, we will work towards building a quasi-isomorphism of $\DG$-algebras
\[
\scrEnd_{\Upgamma}(\PP)\to \scrT.
\]
This requires the following easy, direct, verifications.

\begin{lemma}\label{graded repetition comp ok}
	For $f,g \in \scrEnd_{\Upgamma}(\PP)$, the following hold:
	\begin{enumerate}
		\item\label{graded repetition comp 1}	$\widebar{f\circ g}=\widebar{f}\circ\,\widebar{g}$. 
		\item\label{graded repetition comp 2} $\Updelta(f \circ g)= \Updelta(f)\circ \widebar{g} + (-1)^{|f|} \widebar{f}\circ \Updelta(g)$.
	\end{enumerate}
\end{lemma}
\begin{proof}
For (1), since the complex $\PP$ has precisely $n$ non-zero entries, $f$ and $g$ are encoded by a collection of $n$ maps, say $(\uppsi_i)$ and $(\upvarphi_i)$ respectively, where $\uppsi_i, \upvarphi_i$ have domain $P_i$. Both $\widebar{f\circ g}$ and $\widebar{f}\circ\,\widebar{g}$ are encoded by an infinite number of maps, and we must show that the statement holds for each of these.
	
Choose an arbitrary $P_j$ in the component $\PP[\ell n]$, for some $0\leq j\leq n-1$ and some $\ell\geq 0$.  If $\upvarphi_j=0$ then $\widebar{f\circ g}=\widebar{f}\circ\,\widebar{g}$ out of this $P_j$, since both are zero.  Hence we can assume that $\upvarphi_j\neq 0$, and so since $\PP$ has length $n$, necessarily we can write $\upvarphi_j\colon P_j\to P_k$ for some $0\leq k\leq n-1$.  It follows that $(-1)^{|\ell n||g|}\upvarphi_j$ maps $P_j$ (in the component $\PP[\ell n]$) into the same component $\PP[\ell n]$ of $\scrP$.   Thus, under composition, the sign on  $\widebar{f}$ is $(-1)^{|\ell n||f|}$, and so overall on this $P_j$ we have 
	\begin{align*}
		\widebar{f}\circ\,\widebar{g}&=(-1)^{|\ell n|(|f|+|g|)}\,\uppsi_k\circ \upvarphi_j=\widebar{f\circ g}.
	\end{align*}
	The result follows. 
	
	For (2), 
	\begin{align*}
		\Updelta(f \circ g) &\colonequals \updelta\big(\widebar{f\circ g}\big)-
		\widebar{\big(\updelta(f\circ g)\big)}\\
		&=\updelta\big(\widebar{f}\circ \widebar{g}\big)-
		\widebar{\big(\updelta(f\circ g)\big)}\tag{Part \eqref{graded repetition comp 1}}\\
		&=\updelta\big(\widebar{f}\big)\circ \widebar{g} +(-1)^{|f|}\widebar{f}\circ\updelta\big(\widebar{g}\big)
		-
		\widebar{\big(\updelta(f)\circ g+(-1)^{|f|}f\circ \updelta(g)\big)}\tag{graded Leibniz, twice}\\
		&=\updelta\big(\widebar{f}\big)\circ \widebar{g} +(-1)^{|f|}\widebar{f}\circ\updelta\big(\widebar{g}\big)
		-
		\big(	\widebar{\updelta(f)}\circ \widebar{g}+(-1)^{|f|}\widebar{f}\circ \widebar{\updelta(g)}\big)\tag{Part \eqref{graded repetition comp 1}}\\
		&=\big(\updelta(\widebar{f})-\widebar{\updelta(f)}\big)\circ\widebar{g}
		+(-1)^{|f|}\widebar{f}\circ\big( \updelta(\widebar{g})-\widebar{\updelta(g)} \big)
		\\
		&=\Updelta(f) \circ \widebar{g} + (-1)^{|f|} \widebar{f} \circ \Updelta(g).\tag*{\qedhere}
	\end{align*} 
\end{proof}

\begin{lemma}\label{prep for DG}
	Given $x \in \scrE$, $g \in \scrHom_{\Upgamma}(\PP , \scrP) $, the following hold.
	\begin{enumerate}
		\item\label{prep for DG 1} If $x = y \circ \upsigma$ for some $y\in\scrE$, then $x_0= (y \circ \upsigma)_0=0$.
		\item\label{prep for DG 2} Restriction commutes with the differentials, namely $\updelta(x_0)=\big(\updelta(x)\big)_0$.
		\item\label{prep for DG 3}  $(\widebar{g})_0=g$.
		\item\label{prep for DG 4} $\uptau\circ \widebar{g} = (-1)^{|\upsigma||g|} \widebar{g} \circ \uptau$.
		\item\label{prep for DG 5} $\updelta(\Updelta(g))=-\Updelta(\updelta(g))$.
		\item\label{prep for DG 6} $\Updelta(g) \circ \uptau \circ \upsigma = \Updelta(g)$.
		\item\label{prep for DG 7} If $x$ satisfies  $x \circ \uptau \circ \upsigma = x$ (e.g.\ $x=\Updelta(g)$ by \textnormal{(6)} above), then $x \circ \updelta(\uptau) =0$.
		\item\label{prep for DG 8} If $x\in\sigE$ and $x \circ \uptau \circ \upsigma = x$, then $x\circ \uptau\in\nsigE$.  
	\end{enumerate}
\end{lemma}
\begin{proof}
	Part (1) holds since $\upsigma$ maps $\PP $ to zero, hence $y\circ\upsigma$ maps $\PP$ to zero.  Parts (2) and (3) are immediate consequences of the definitions. For Part (4), $g$ is encoded by finitely many morphisms $(\upvarphi_i)_{i=0}^{n-1}$, where each $\upvarphi_i$ has domain $P_i$.  To show the statement, consider the complex $\scrP$, and choose an arbitrary $P_j$ in the component $\PP[\ell n]$, for some $0\leq j\leq n-1$ and some $\ell\geq 0$. If $\upvarphi_j=0$, then $\widebar{g}$ restricted to all $P_j$ is zero, hence $\uptau\circ \widebar{g} = 0=(-1)^{|\upsigma||g|} \widebar{g} \circ \uptau$ holds on $P_j$.  Else $\upvarphi_j\neq 0$, in which case $\upvarphi_j$ lands in the non-zero part of the complex $\scrP$.  As $\uptau$ points to the left, the fact that $\uptau\circ \widebar{g} = (-1)^{|\upsigma||g|} \widebar{g} \circ \uptau$ holds on $P_j$ is then clear.
	
Part (5) follows since $\updelta^2=0$, and so using this twice gives 
	\[
	\updelta\big(\Updelta(g)\big)=\updelta\big(\updelta (\widebar{g}) - \widebar{\updelta(g)}\big)=-\updelta\big(\, \widebar{\updelta(g)}\,\big)=-\updelta\big(\, \widebar{\updelta(g)}\,\big)+\widebar{\updelta(\updelta(g))}=-\Updelta(\updelta(g)).
	\]
	Part (6) holds since by \eqref{sig tau basic} $(\Id -\uptau \circ \upsigma)$ is the identity on $\PP$ and is zero elsewhere, and furthermore $\updelta (\widebar{g})$ and $\widebar{\updelta(g)}$ agree on $\PP $, since there they both equal $\updelta(g)$.  Hence $\Updelta(g)\circ (\Id - \uptau \circ \upsigma) =0$.  Part (7) is
	\begin{align*}
		x \circ \updelta(\uptau) & = x \circ \uptau \circ \upsigma \circ \updelta(\uptau) \tag{by assumption on $x$}\\
		&= (-1)^{|\upsigma|} x \circ \uptau \circ \updelta(\upsigma \circ \uptau) \tag{Leibniz Rule, and $\updelta (\upsigma) = 0$ by \ref{Lemma: deltasigma is zero}}\\
		&= (-1)^{|\upsigma|} x \circ \uptau \circ \updelta(\Id)\tag{$\upsigma \circ \uptau  = \Id$ by \eqref{sig tau basic}} \\
		&=0.
	\end{align*}
	For Part (8), observe that
	\begin{align*}
		\upsigma \circ x \circ \uptau & = (-1)^{|\upsigma||x|} x \circ \upsigma \circ \uptau \tag{$x\in\sigE$ by assumption} \\
		& = (-1)^{|\upsigma||x|} x \tag{$\upsigma \circ \uptau  = \Id$ by \eqref{sig tau basic}}  \\
		& = (-1)^{|\upsigma|(|x|+|\uptau|+1)} x \tag{$|\upsigma||\uptau| \equiv_2 -|\upsigma|$, since $-n^2 \equiv_2 -n$}  \\  
		& = (-1)^{|\upsigma|(|x|+|\uptau|+1)} x \circ \uptau \circ \upsigma, \tag{by assumption on $x$}
	\end{align*} 
	and so $x\circ \uptau\in\nsigE$.
\end{proof}

The above result holds for all $g\in\scrHom_{\Upgamma}(\PP,\scrP)$.  When $g\in\scrEnd_\Upgamma(\PP)$,  we can say more.

\begin{lemma}\label{prep for DG U}
	If $g \in\scrEnd_\Upgamma(\PP)$, then the following hold.
	\begin{enumerate}
		\item\label{prep for DG U 1} $\widebar{g}\in\sigE$.
		\item\label{prep for DG U 2} $\Updelta(g)\in\sigE$. 
		\item\label{prep for DG U 3} $\Updelta(g) \circ \uptau\in\nsigE$.  
	\end{enumerate}
\end{lemma}
\begin{proof}
For Part (1), by \ref{prep for DG}\eqref{prep for DG 3} $(\upsigma\circ \widebar{g})_0=\upsigma\circ (\widebar{g})_0=\upsigma\circ g$, thus since  $g$ takes $\PP$ to $\PP$ and $\upsigma$ points to the right, $(\upsigma\circ \widebar{g})_0=0$. 
By \ref{prep for DG}\eqref{prep for DG 1}, $(\widebar{g}\circ\upsigma)_0=0$.  Combining, we see that $\upsigma\circ \widebar{g} =(-1)^{|\upsigma||g|} \widebar{g} \circ \upsigma$ holds on $\PP$, since both equal $0$.

Since $g\in\scrEnd_\Upgamma(\PP)$, the only non-zero morphisms that constitute $\upiota \circ g$ take $\PP$ to $\PP$, hence the only non-zero morphisms that constitute $\widebar{g}$ map components $\PP[\ell n]$ to themselves. We finally claim that $\upsigma\circ \widebar{g} =(-1)^{|\upsigma||g|} \widebar{g} \circ \upsigma$ holds on all components $\PP[\ell n]$ with $\ell>0$. Since the only non-zero morphisms that constitute $\widebar{g}$ map $\PP[\ell n]$ to itself, $\upsigma$ points to the right, and $\ell>0$, $\upsigma\circ \widebar{g} =(-1)^{|\upsigma||g|} \widebar{g} \circ \upsigma$ follows.
	
For Part (2), $\widebar{g}\in\sigE$ by (1), so since $\sigE$ is closed under the differential by \ref{Lemma: sigE is DG}\eqref{Lemma: sigE is DG 2}, necessarily $\updelta(\widebar{g})\in\sigE$.  Also, $\updelta(g)\in\scrEnd_\Upgamma(\PP)$, 
    so applying (1) to $\updelta(g)$ gives $\widebar{\updelta(g)}\in\sigE$.  Combining, $\Updelta(g)=\updelta (\widebar{g}) - \widebar{\updelta(g)}\in\sigE$.  
	
	Part (3) is just \ref{prep for DG}\eqref{prep for DG 8}, using (2) together with \ref{prep for DG}\eqref{prep for DG 6} to see that the two assumptions on $x=\Updelta(g)$ hold.
\end{proof}

\subsection{The quasi-isomorphism}\label{subsec:qisTtoEnd}
The following is the main result of this subsection.

\begin{thm}\label{thm:qis of scrT}
	Under Setup~\textnormal{\ref{key setup}} and Notation~\textnormal{\ref{notation for DG morph}}, the following are morphisms of $\DG$ $\mathbb{Z}$-modules
	\[
	\begin{array}{rcl}
		\scrT & \longleftrightarrow&  \scrEnd_\Upgamma(\PP) \\
		(x,y) & \mapsto &x_0\\
		(\widebar{g}, \,(-1)^{|g|}\Updelta(g) \circ \uptau)& \mapsfrom&  g
	\end{array}
	\]
	which induce bijections on cohomology. Moreover, the bottom map is a morphism of $\DG$-algebras. 
\end{thm}
\begin{proof}
Write $\mathbb{F}$ for the top map, and $\mathbb{G}$ for the bottom map.  The map $\mathbb{F}$ is well defined since by \ref{def U prop}, the restiction $x_0$ belongs to $\scrEnd_\Upgamma(\PP)$. The map $\mathbb{G}$ is well defined since $\widebar{g} \in \sigE$ by \ref{prep for DG U}\eqref{prep for DG U 1} and $(-1)^{|g|} \Updelta(g) \circ \uptau \in \nsigE$ by \ref{prep for DG U}\eqref{prep for DG U 3}.  Thus both $\mathbb{F}$ and $\mathbb{G}$ are well defined, and they are clearly morphisms of graded $\mathbb{Z}$-modues.
	
	We now claim that $\mathbb{F}$ and $\mathbb{G}$ respect the differentials.  This is true for $\mathbb{F}$ since
	\begin{align*}
		\mathbb{F}(\upxi(x,y)) = (\updelta(x) - (-1)^{|x|}y \circ \upsigma )_0 
		\stackrel{\scriptstyle\ref{prep for DG}\eqref{prep for DG 1}}{=}\updelta(x)_0
		\stackrel{\scriptstyle\ref{prep for DG}\eqref{prep for DG 2}}{=} \updelta(x_0) = 
		\updelta\, \mathbb{F} (x,y),
	\end{align*}
	and it is true for $\mathbb{G}$ since
	\begin{align*}
		\upxi\, \mathbb{G}(g)  & = \upxi( \widebar{g}, (-1)^{|g|}\Updelta(g) \circ \uptau) \\
		& = \big( \updelta(\widebar{g}) -(-1)^{|g|+|g|} \Updelta(g) \circ \uptau \circ \upsigma, (-1)^{|g|} \updelta(\Updelta(g) \circ \uptau)\big) \\
		& = \big( \updelta(\widebar{g}) -\Updelta(g) , (-1)^{|g|} \updelta(\Updelta(g)) \circ \uptau + (-1)^{|g|+|g|+1} \Updelta(g) \circ \updelta(\uptau) \big)\tag{by \ref{prep for DG}\eqref{prep for DG 6}} \\
		& =\big( \updelta(\widebar{g}) -(\updelta(\widebar{g}) - \widebar{\updelta(g)} ), (-1)^{|g|} \updelta(\Updelta(g)) \circ \uptau\big) \tag{by \ref{prep for DG}\eqref{prep for DG 7}}\\
		& =( \widebar{\updelta(g)} ,  (-1)^{|\updelta(g)|}\Updelta(\updelta(g)) \circ \uptau ) \tag{by \ref{prep for DG}\eqref{prep for DG 5}}\\
		&= \mathbb{G}(\updelta(g)).
	\end{align*}

	Next, we claim that $\mathbb{F}$ and $\mathbb{G}$ restrict to an inverse bijection on cohomology.  It is clear that $\mathbb{F}( \mathbb{G}(g)) = (\widebar{g})_0=g$ by \ref{prep for DG}\eqref{prep for DG 3}, and so we just need to establish that $\mathbb{G}\circ\mathbb{F}$ is the identity  when restricted to cohomology. 
	
	Now, any $x \in \sigE$ can be written in the form 
	\[
	x =  \widebar{x_0} + (x  - \widebar{x_0} ) = \widebar{x_0} + (x  - \widebar{x_0} ) \circ \Id_{\ge n} 
	\stackrel{{\scriptstyle\eqref{sig tau basic}}}{=} 
	\widebar{x_0} + (x  - \widebar{x_0} ) \circ \uptau \circ \upsigma\label{split x complex}
	\]
	since $\widebar{x_0}$ and $x$ agree on $\PP $.  Hence $(x-\widebar{x_0})\circ \uptau \circ \upsigma=(x-\widebar{x_0})$, so by \ref{prep for DG}\eqref{prep for DG 7}, $(x-\widebar{x_0}) \circ \updelta(\uptau)=0$.   Furthermore, as $x\in\sigE$ by assumption, $\widebar{x_0}\in\sigE$ by combining \ref{def U prop} 
	 and \ref{prep for DG U}\eqref{prep for DG U 1}.  Thus $(x-\widebar{x_0})\in\sigE$.  Hence, by \ref{prep for DG}\eqref{prep for DG 8} applied to $x-\widebar{x_0}$, we deduce that $(x-\widebar{x_0})\circ \uptau\in\nsigE$.
	
	Hence $z\colonequals (0,  (-1)^{|x|}(x  - \widebar{x_0}) \circ \uptau ) \in \scrT$ is an element of degree $|x|-1$ such that
	\begin{align*}
		\upxi(z )& =\left( -(-1)^{|x|-1}(-1)^{|x|}(x  - \widebar{x_0}) \circ \uptau\circ \upsigma  ,  (-1)^{|x|}\updelta((x  - \widebar{x_0}) \circ \uptau )\right) \\
		& =\left((x  - \widebar{x_0})  ,  (-1)^{|x|}\updelta(x  - \widebar{x_0})\circ \uptau + (-1)^{|x|+|x|} (x-\widebar{x_0}) \circ \updelta(\uptau) \right) \\
		& =\left((x  - \widebar{x_0}) ,  (-1)^{|x|}\updelta(x  - \widebar{x_0})\circ \uptau\right) \\
		& =\left((x  - \widebar{x_0}),  (-1)^{|x|}\updelta(x) \circ \uptau - (-1)^{|x|} \updelta(\widebar{x_0})\circ \uptau \right).
	\end{align*}
	Using this we then calculate 
	\begin{align*}
		\left( x, (-1)^{|x|} \updelta(x) \circ \uptau  \right)  
		& = \left( \widebar{x_0},  (-1)^{|x|} \updelta(\widebar{x_0})\circ \uptau \right)  +  \upxi(z) \\
		& = \mathbb{G}( x_0 ) +  \upxi(z).
	\end{align*}
	
	Now an element $(x,y) \in \scrT$ belongs to the kernel of $\upxi$ if and only if
	\[
	\updelta(x) - (-1)^{|x|} y \circ \upsigma =0 \quad \text{ and } \quad\updelta(y)=0.
	\]
	Right multiplying the first identity by $\uptau$, and recalling that $\upsigma \circ \uptau = \Id$ by \eqref{sig tau basic}, we deduce that for $(x,y)$ to be in the kernel of $\upxi$ it is necessary that $y=(-1)^{|x|}\updelta(x) \circ \uptau$. Hence if $(x,y) \in \Ker \upxi$, then $(x,y)=(x, (-1)^{|x|}\updelta(x) \circ \uptau)$, so by the above
	\begin{align*}
		\mathbb{G}(\mathbb{F}(x,y))&= \mathbb{G}(x_0) = (x,y)- \upxi(z).
	\end{align*}
	As such, $\mathbb{G}$ and $\mathbb{F}$ induce a bijection on cohomology. 
	
	Finally, to show that  $\mathbb{G}$ is a morphism of $\DG$-algebras, we just need to check it preserves the products. This follows since
		\begin{align*}
			\mathbb{G}(f) \cdot\mathbb{G}(g) & = \big(\widebar{f}, (-1)^{|f|}\Updelta(f) \circ \uptau ) \cdot (\widebar{g}, (-1)^{|g|}\Updelta(g) \circ \uptau\big) \\
			& = \big(\widebar{f\circ g}, (-1)^{|g|} \widebar{f} \circ \Updelta(g) \circ \uptau + (-1)^{|g|(|\upsigma|+1)+|f|} \Updelta(f) \circ \uptau \circ \widebar{g}\big) \tag{by definition of $\cdot$, and \ref{graded repetition comp ok}\eqref{graded repetition comp 1}}\\
			& = \big(\widebar{f\circ g}, (-1)^{|g|} \widebar{f} \circ \Updelta(g) \circ \uptau + (-1)^{|g|(|\upsigma|+1)+|f|+|\upsigma||g|} \Updelta(f) \circ  \widebar{g} \circ \uptau\big)  \tag{by \ref{prep for DG}\eqref{prep for DG 4}} \\
			& = \big(\widebar{f\circ g}, (-1)^{|g|} \widebar{f} \circ \Updelta(g) \circ \uptau + (-1)^{|f|+|g|} \Updelta(f) \circ  \widebar{g} \circ \uptau\big)\\
			& = \big(\widebar{f\circ g}, (-1)^{(|f|+|g|)}(\Updelta(f)\circ \widebar{g}  + (-1)^{|f|}  \widebar{f} \circ \Updelta(g) )\circ \uptau\big)  \\
			& = \big(\widebar{f\circ g}, (-1)^{(|f|+|g|)}\Updelta(f\circ g)\circ \uptau\big) \tag{by \ref{graded repetition comp ok}\eqref{graded repetition comp 2}}  \\
			& = \mathbb{G}(f \circ g).\qedhere
		\end{align*}
\end{proof}

\subsection{Computation of cohomology}\label{sec.cohomology_end}
In this section we compute the cohomology of the $\DG$-algebra $\scrEnd_\Upgamma(\PP)$, and thus via \ref{thm:qis of scrT} the cohomology of $\scrT$.  Again, consider the complex $\PP$ of projectives, which induces an exact sequence
\[
0 \to M \xrightarrow{\upbeta}  P_{n-1} \to P_{n-2} \to \dots \to P_{1} \to P_0 \xrightarrow{\upalpha} M  \to 0.
\]
Using the connecting map $d_0\colonequals \upbeta \circ \upalpha$, accounting for signs correctly as in \ref{ex:spelloutscrP}, we obtain the infinite projective resolution $\scrP$ of the ${\Upgamma}$-module $M$.

As in \S\ref{general qis section}, write  $\upiota$ for the inclusion $\upiota \colon \PP \to \scrP$. This induces a short exact sequence of complexes
\[
0 \to \PP \xrightarrow{\upiota} \scrP \xrightarrow{\uppi}\scrP_{\leq n }\to 0,
\]
where the complex $\scrP_{\leq n} $ is the brutal truncation of $\scrP$, as illustrated in the following diagram.
\begin{equation}
\begin{tikzcd}[column sep=15pt]
\PP\arrow[d, "\upiota"]&  &     &   & P_{n-1} \arrow[r] \arrow[d,equals] & P_{n-2} \arrow[r] \arrow[d,equals] & \hdots \arrow[r] \arrow[d,equals] & P_0 \arrow[d,equals] \\
\scrP\arrow[d,"\uppi"]& \hdots \arrow[d,equals] \arrow[r] & P_{1} \arrow[d,equals] \arrow[r] & P_0 \arrow[d,equals] \arrow[r] & P_{n-1} \arrow[r]                    & P_{n-2} \arrow[r]                    & \hdots \arrow[r]                    & P_0                    \\
\scrP_{\leq n}&\hdots \arrow[r]  & P_{1} \arrow[r]  & P_0  &     &   &  & 
\end{tikzcd}\label{eqn:uppidef}
\end{equation}
Notice that the construction of $\scrP$, in particular the sign choices spelled out in \ref{ex:spelloutscrP}, together with the fact that $[1]$ shifts to the left with a sign, implies that there is an equality of complexes 
\begin{equation}\label{eq.P-less-n}
\scrP_{\leq n }[-n] = \scrP.
\end{equation}

\begin{lemma}\label{lem.map-pi}
		Given $M$, $\PP$ and $\scrP$ as above, let $N$ be a $\Upgamma$-module, and consider the map $\circ\uppi \colon \scrHom_{\Upgamma}(\scrP_{\leq n}, N) \to \scrHom_{\Upgamma}(\scrP, N)$.
		\begin{enumerate}
			\item\label{lem.map-pi 1} The map
			\[ 	
			\Hom_{\Upgamma}(M,N) \to \Ext^{n}_{\Upgamma}(M,N)
			\]
			induced by $\HH^{n}(\circ\uppi)$ is surjective, and has kernel $\image(\circ d_0) \subseteq \Hom_{\Upgamma}(P_0,N)$.
			\item\label{lem.map-pi 2}  For every $i >0$, the map
			\[
			\Ext^{i}_{\Upgamma}(M,N) \to \Ext^{i+n}_{\Upgamma}(M,N)
			\]
			induced by $\HH^{i+n}(\circ\uppi)$ is an isomorphism.
		\end{enumerate}
	\end{lemma}
\begin{proof}
Since $\scrP$ is a projective resolution of $M$, we can use the complex $\scrHom_\Upgamma(\scrP,N)$ to calculate the groups $\Ext^i_\Upgamma(M,N)$, since with our conventions on grading, 
\[
\HH^i(\scrHom_\Upgamma(\scrP,M))\cong \Ext^i(M,N).
\] 
By \eqref{eq.P-less-n}, we can also use $\scrHom_\Upgamma(\scrP_{\leq n},N)$ to calculate the same (shifted) Ext groups, via $\Ext^i_\Upgamma(M,N)\cong \HH^{i+n}(\scrHom_{\Upgamma}(\scrP_{\leq n}, N)) $.
		
Applying $\Hom_\Upgamma(-,N)$ to the bottom half of \eqref{eqn:uppidef} gives a commutative diagram
\begin{equation}
\begin{array}{c}
\begin{tikzpicture}
\node (a1) at (7.5,0.5) {${}_{\Upgamma}(P_0,N)$};
\node (a2) at (9.5,0.5) {${}_{\Upgamma}(P_{1},N)$};
\node (a3) at (11,0.5) {$\hdots$};
\node (b4) at (5.15,-1) {${}_{\Upgamma}(P_{n-1},N)$};
\node (b5) at (7.5,-1) {${}_{\Upgamma}(P_{0},N)$};
\node (b6) at (9.5,-1) {${}_{\Upgamma}(P_{1},N)$};
\node (b7) at (11,-1) {$\hdots$,};
\draw[->] (a1) -- node[above]{$\scriptstyle \circ d_1 $}(a2);
\draw[->] (a2) -- node[above]{$\scriptstyle $}(a3);
\draw[->] (b4) -- node[above]{$\scriptstyle \circ d_0 $}(b5);
\draw[->] (b5) -- node[above]{$\scriptstyle \circ d_1 $}(b6);
\draw[->] (b6) -- node[above]{$\scriptstyle $}(b7);
\draw[double equal sign distance] (a1) -- (b5);
\draw[double equal sign distance] (a2) -- (b6);				
\end{tikzpicture}
\end{array}\label{eqn:forHn}
\end{equation}
where ${}_{\Upgamma}(P_k,N)$ is shorthand for $\Hom_{\Upgamma}(P_k, N)$, and ${}_{\Upgamma}(P_0,N)$ is in degree $n$.  It is clear from \eqref{eqn:forHn} that $\HH^{j}(\circ\uppi)$ is the identity when $j>n$, and $\HH^j(\circ\uppi)$ is zero if $j<n$.  In particular, if $i>0$, then $\HH^{i+n}(\circ\uppi)=\mathrm{Id}$ and so it induces morphisms
		\[
		\Ext^i_\Upgamma(M,N)\cong \HH^{i+n}(\scrHom_{\Upgamma}(\scrP_{\leq n}, N)) \xrightarrow{\mathrm{Id}}\HH^{i+n}(\scrHom_{\Upgamma}(\scrP, N))\cong \Ext^{i+n}_{\Upgamma}(M,N).
		\]
		The composition is an isomorphism, proving (2).  
		
		For (1), notice from \eqref{eqn:forHn} that $\HH^{n}(\circ\uppi)$ is the induced map
		\[
		\Ker(\circ d_1)\to \Ker(\circ d_1)/\image(\circ d_0).
		\]
		Since this is induced by the identity, it is clearly surjective with claimed kernel.
	\end{proof}

Consider now the following commutative diagram, where all the rows and columns are exact, because $\PP, \scrP$ and $\scrP_{\leq n}$ are bounded-above complexes of projective modules:
\begin{center}
\begin{tikzcd}
& 0 \arrow[d]     & 0 \arrow[d]     & 0 \arrow[d]        &   \\
0 \arrow[r] & {\scrHom_{\Upgamma}(\scrP_{\leq n}, \PP)} \arrow[r, "\circ\uppi"] \arrow[d, "\upiota\circ"]  & {\scrHom_{\Upgamma}(\scrP, \PP)} \arrow[r, "\circ\upiota"] \arrow[d, "\upiota\circ"] & \scrEnd_{\Upgamma}(\PP) \arrow[r] \arrow[d, "\upiota\circ"]           & 0 \\
0 \arrow[r] & {\scrHom_{\Upgamma}(\scrP_{\leq n}, \scrP)} \arrow[r, "\circ\uppi"] \arrow[d, "\uppi\circ"] & \scrEnd_{\Upgamma}(\scrP) \arrow[r, "\circ\upiota"] \arrow[d, "\uppi\circ"]   & {\scrHom_{\Upgamma}(\PP, \scrP)} \arrow[r] \arrow[d, "\uppi\circ"] & 0 \\
0 \arrow[r] & \scrEnd_{\Upgamma}(\scrP_{\leq n }) \arrow[r, "\circ\uppi"] \arrow[d]      & {\scrHom_{\Upgamma}(\scrP,\scrP_{\leq n})} \arrow[r, "\circ\upiota"] \arrow[d]    & {\scrHom_{\Upgamma}(\PP,\scrP_{\leq n})} \arrow[r] \arrow[d]     & 0 \\
			& 0                                                                               & 0                                                                         & 0.                                                             &  
\end{tikzcd}
\end{center}
This diagram will allow us to compute the cohomology of $\scrEnd_{\Upgamma}(\PP)$ and of $\scrHom_{\Upgamma}(\PP, \scrP)$.  As in the proof of \ref{lem.map-pi}, the map
\[ 
\Ext^{i-n}_\Upgamma(M,M)\cong\HH^{i}(\scrHom_{\Upgamma}(\scrP_{\leq n}, \scrP)) \to \HH^{i}(\scrEnd_{\Upgamma}(\scrP)) \cong \Ext^{i}_\Upgamma(M,M) 
\]
is an isomorphism for $i > n$, and is surjective for $i=n$, with kernel $\image(\circ d_0)\subseteq\Hom_\Upgamma(P_0,M)$.
Likewise, the map
\[ 
\Ext^{i}_\Upgamma(M,M) \cong\HH^{i}(\scrEnd_{\Upgamma}(\scrP_{\leq n})) \to \HH^{i}(\scrHom_{\Upgamma}(\scrP, \scrP_{\leq n})) \cong \Ext^{i+n}_\Upgamma(M,M) 
\]
is an isomorphism for $i > 0$, and is surjective for $i=0$, with kernel $\image(\circ d_0)\subseteq\Hom_\Upgamma(P_0,M)$.

Notice that we cannot obtain the same kind of information for the map induced by $\circ\uppi$ in the first row of the diagram, because the complexes in the first row do not calculate Ext groups, since $\PP$ is not a projective resolution. Therefore, to calculate the underlying graded vector space of the cohomology of $\scrEnd_\Upgamma(\PP)$, we proceed by calculating the cohomology of the complexes in the last column in the second and third rows, using the information we know about the maps induced by $\circ\uppi$ in the second and third row. 
Thus we obtain from the second row an isomorphism of vector spaces
\begin{equation}\label{eq.coh-ext}
\HH^{i}(\scrHom_{\Upgamma}(\PP, \scrP)) = \begin{cases} 0 &\quad i < 0,\ i \geq n\\ 
\Ext^i_{\Upgamma}(M,M) &\quad i=0,  \hdots , n-2 \\ \Ext^{n-1}_{\Upgamma}(M,M) \oplus \image(\circ d_0) &\quad i =n-1,
 \end{cases}
\end{equation}
and from the third row an isomorphism of vector spaces
\begin{equation}
\HH^{i}(\scrHom_{\Upgamma}(\PP,\scrP_{\leq n})) =\begin{cases} 0 &\quad i < -n,\ i \geq 0\\ 
\Ext^{i+n}_{\Upgamma}(M,M) &\quad i=-n,  \hdots , -2\\ \Ext^{n-1}_{\Upgamma}(M,M) \oplus \image(\circ d_0) &\quad i =-1. 
\end{cases}\label{eq.coh-ext2}
\end{equation}
From these, using the third column in the diagram above, we can calculate the cohomology of $\scrEnd_{\Upgamma}(\PP)$, which, by \ref{thm:qis of scrT}, coincides with the cohomology of $\scrT$. In fact, from the third column one obtains a long exact sequence in cohomology relating $\HH^*(\scrEnd_\Upgamma(\PP))$, \eqref{eq.coh-ext} and \eqref{eq.coh-ext2}. From the form of \eqref{eq.coh-ext} and \eqref{eq.coh-ext2}, one notices that the map induced by $\uppi\circ$ in cohomology is always zero, because either the source or the target is zero. This immediately implies the following.

\begin{thm}\label{eq.coh-t}
The only non-zero cohomology groups of the $\DG$-algebras $\scrT$ and $\scrEnd_{\Upgamma}(\PP)$ are the following, where in each degree we take the direct sum of the vector spaces stated.
\[
\begin{tikzpicture}
\node (a2) at (-4.6,0) {\scalebox{0.8}{$\HH^{-n+1}(\scrT)$}};
\node (a3) at (-3.35,-0.1) {\scalebox{0.8}{$\hdots$}};
\node (a4) at (-2.1,0) {\scalebox{0.8}{$\HH^{-1}(\scrT)$}};
\node (a5) at (0,0) {\scalebox{0.8}{$\HH^{0}(\scrT)$}};
\node (a6) at (2,0) {\scalebox{0.8}{$\HH^{1}(\scrT)$}};
\node (a7) at (3.1,-0.1) {\scalebox{0.8}{$\hdots$}};
\node (a8) at (4.4,0) {\scalebox{0.8}{$\HH^{n-2}(\scrT)$}};
\node (a8) at (6.5,0) {\scalebox{0.8}{$\HH^{n-1}(\scrT)$}};

\node (b2) at (-4.6,-0.75) {\scalebox{0.8}{$\Hom_{\Upgamma}(M,M)$}};
\node (b3) at (-3.35,-0.85) {\scalebox{0.8}{$\hdots$}};
\node (b4) at (-2.1,-0.75) {\scalebox{0.8}{$\Ext^{n-2}_{\Upgamma}(M,M)$}};
\node (b5) at (0,-0.75) {\scalebox{0.8}{$\Ext^{n-1}_{\Upgamma}(M,M)$}};
\node (b5b) at (0,-1.25) {\scalebox{0.8}{$\image(\circ d_0)$}};
\node (b5c) at (0,-1.75) {\scalebox{0.8}{$\Hom_\Upgamma(M,M)$}};
\node (b6) at (2,-1.75) {\scalebox{0.8}{$\Ext^1_\Upgamma(M,M)$}};
\node (b7) at (3.1,-1.85) {\scalebox{0.8}{$\hdots$}};
\node (b8) at (4.4,-1.75) {\scalebox{0.8}{$\Ext^{n-2}_\Upgamma(M,M)$}};
\node (b9c) at (6.5,-1.75) {\scalebox{0.8}{$\Ext^{n-1}_\Upgamma(M,M)$}};
\node (b9b) at (6.5,-1.25) {\scalebox{0.8}{$\image(\circ d_0)$}};

\end{tikzpicture}
\]
\end{thm}

\section{Yoneda and Uniqueness}\label{sec:Yoneda}
This section is concerned with basic properties of periodic resolutions, in terms of Yoneda extension classes, and from this answers questions on the dependence of the $\DG$-algebras $\scrT$ and $\scrEnd_\Upgamma(\PP)$ on the choice of input data.

\medskip
Let ${\Upgamma}$ be a $\mathbb{C}$-algebra and let $M$ be a periodic $\Upgamma$-module, as in \ref{key setup}.  We retain the notation there, in particular the complex $\PP$ of finitely generated projectives which induces an exact sequence 
\begin{equation}
0 \to M \xrightarrow{\upbeta}  P_{n-1} \to P_{n-2} \to \dots \to P_{1} \to P_0 \xrightarrow{\upalpha} M  \to 0.\label{eq:Yon1}
\end{equation}

\begin{remark}\label{rem.uniqueness}
Let $\PP,\PP'$ be two periodic projective resolutions of $M$, as in \eqref{eq:Yon1}, of the same length. 
\begin{enumerate}
\item\label{rem.uniqueness1} There exists a chain map $g \colon \PP \to\PP'$ lifting the identity on $M$, which is unique up to chain homotopy in the following sense: given any other chain map $h \colon  \PP \to \PP'$ lifting the identity on $M$ there exists a chain homotopy $k$ such that
\[
g_i-h_i=
\begin{cases}
d_{1}' \circ k_0&\mbox{if } i=0\\
d_{i+1}'\circ k_i + k_{i-1}\circ d_i&\mbox{if } 1\leq i\leq n-2\\
 \upbeta'\circ k_{n-1} + k_{n-2}\circ d_{n-1}&\mbox{if } i=n-1\\
\end{cases}
\]
as illustrated in the following diagram
\[
\begin{tikzpicture}[xscale=0.75,yscale=1.5]
\node (A1) at (-1.75,0) {$P_{n-1}$};
\node (A2) at (0.75,0) {$P_{n-2}$};
\node (A3) at (3,0) {$\hdots$};
\node (A4) at (5,0) {$P_1$};
\node (A5) at (7,0) {$P_0$};
\draw[->] (A1)--node[above]{$\scriptstyle d_{n-1}$}(A2);
\draw[->] (A2)--(A3);
\draw[->] (A3)--(A4);
\draw[->] (A4)--node[above]{$\scriptstyle d_1$}(A5);
\node (Bm1) at (-5.5,-1) {$0$};
\node (B0) at (-4,-1) {$M$};
\node (B1) at (-1.75,-1) {$P'_{n-1}$};
\node (B2) at (0.75,-1) {$P'_{n-2}$};
\node (B3) at (3,-1) {$\hdots$};
\node (B4) at (5,-1) {$P'_1$};
\node (B5) at (7,-1) {$P'_0$};
\draw[->] (Bm1)--(B0);
\draw[->] (B0)--node[below]{$\scriptstyle\upbeta'$}(B1);
\draw[->] (B1)--node[below]{$\scriptstyle d_{n-1}'$}(B2);
\draw[->] (B2)--(B3);
\draw[->] (B3)--(B4);
\draw[->] (B4)--node[below]{$\scriptstyle d_1'$}(B5);
\draw[->] (A1) --(B1);
\draw[->] (A2) -- node[right] {$\scriptstyle $}(B2);
\draw[->] (A4) -- node[gap] {$\scriptstyle g_{1}-h_1$} (B4);
\draw[->] (A5) -- node[gap] {$\scriptstyle g_{0}-h_0$} (B5);
\draw[->] (A1) --node[gap] {$\scriptstyle k_{n-1}$}(B0);
\draw[->] (A2) -- node[gap] {$\scriptstyle k_{n-2}$}(B1);
\draw[->] (A4) -- node[gap] {$\scriptstyle k_{1}$} (B3);
\draw[->] (A5) -- node[gap] {$\scriptstyle k_{0}$} (B4);

\end{tikzpicture}
\]
\item It is important to note that in \eqref{rem.uniqueness1}, the chain map $g \colon \PP \to\PP'$ that lifts the identity on $M$ is not necessarily a quasi-isomorphism. This is because the map 
\[
\begin{tikzpicture}[xscale=0.75,yscale=1.5]
\node (Am1) at (-5.5,0) {$0$};
\node (A0) at (-4,0) {$M$};
\node (A1) at (-1.75,0) {$P_{n-1}$};
\draw[->] (Am1)--(A0);
\draw[->] (A0)--node[above]{$\scriptstyle\upbeta$}(A1);
\draw[->] (A1)--node[above]{$\scriptstyle d_{n-1}$}(A2);
\node (Bm1) at (-5.5,-1) {$0$};
\node (B0) at (-4,-1) {$M$};
\node (B1) at (-1.75,-1) {$P'_{n-1}$};
\draw[->] (Bm1)--(B0);
\draw[->] (B0)--node[above]{$\scriptstyle\upbeta'$}(B1);
\draw[->] (B1)--node[above]{$\scriptstyle d_{n-1}'$}(B2);
\draw[densely dotted,->]  (A0) --node[left] {$\scriptstyle $}(B0);
\draw[->] (A1) --node[left] {$\scriptstyle g_{n-1}$}(B1);
\end{tikzpicture}
\]
 induced by $g_{n-1}$ is not necessarily an isomorphism.  We will show in \ref{lem.qisom-endomo-simple} that this problem can be overcome in the case when $M$ is simple.
\item\label{rem.uniqueness3} Regardless of the above two points, which both involve lifting the identity of $M$, if there happens to exist a quasi-isomorphism $g \colon \PP \to \PP'$,  then $\scrEnd_{\Upgamma}(\PP)$ and $\scrEnd_{\Upgamma}(\PP')$  are quasi-isomorphic  $\DG$-algebras.  This follows by e.g.\ \cite[4.4]{ST}, since $g$ is a homotopy equivalence by the projectivity of all the modules involved. 
\end{enumerate}
\end{remark}

A periodic projective resolution of $M$ may alternatively be thought of as an $n$-fold extension of $M$ with itself formed by projective modules. In particular, according to the theory of $n$-fold extensions of Yoneda \cite{Yo}, it gives a class $\upvartheta \in \Ext^n_{\Upgamma}(M,M)$.

\begin{defin}\label{def.periodic-qiso}
Let $\PP, \PP'$ be two periodic projective  resolutions of $M$, as in \eqref{eq:Yon1}. A periodic quasi-isomorphism $g \colon \PP \to \PP'$ is a collection of maps which combine to give the following commutative diagram.
\[
\begin{tikzpicture}[xscale=0.75,yscale=1.5]
\node (Am1) at (-5.5,0) {$0$};
\node (A0) at (-4,0) {$M$};
\node (A1) at (-1.75,0) {$P_{n-1}$};
\node (A2) at (0.75,0) {$P_{n-2}$};
\node (A3) at (3,0) {$\hdots$};
\node (A4) at (5,0) {$P_1$};
\node (A5) at (7,0) {$P_0$};
\node (A6) at (9,0) {$M$};
\node (A7) at (10.5,0) {$0$};
\draw[->] (Am1)--(A0);
\draw[->] (A0)--node[above]{$\scriptstyle\upbeta$}(A1);
\draw[->] (A1)--(A2);
\draw[->] (A2)--(A3);
\draw[->] (A3)--(A4);
\draw[->] (A4)--(A5);
\draw[->] (A5)--node[above]{$\scriptstyle\upalpha$}(A6);
\draw[->] (A6)--(A7);
\node (Bm1) at (-5.5,-1) {$0$};
\node (B0) at (-4,-1) {$M$};
\node (B1) at (-1.75,-1) {$P'_{n-1}$};
\node (B2) at (0.75,-1) {$P'_{n-2}$};
\node (B3) at (3,-1) {$\hdots$};
\node (B4) at (5,-1) {$P'_1$};
\node (B5) at (7,-1) {$P'_0$};
\node (B6) at (9,-1) {$M$};
\node (B7) at (10.5,-1) {$0$};
\draw[->] (Bm1)--(B0);
\draw[->] (B0)--node[above]{$\scriptstyle\upbeta'$}(B1);
\draw[->] (B1)--(B2);
\draw[->] (B2)--(B3);
\draw[->] (B3)--(B4);
\draw[->] (B4)--(B5);
\draw[->] (B5)--node[above]{$\scriptstyle\upalpha'$}(B6);
\draw[->] (B6)--(B7);
\draw[double distance=2pt]  (A0) --(B0);
\draw[->] (A1) --node[right] {$\scriptstyle g_{n-1}$}(B1);
\draw[->] (A2) -- node[right] {$\scriptstyle g_{n-2}$}(B2);
\draw[->] (A4) -- node[right] {$\scriptstyle g_{1}$} (B4);
\draw[->] (A5) -- node[right] {$\scriptstyle g_{0}$} (B5);
\draw[double distance=2pt] (A6) --(B6);
\end{tikzpicture}
\]
\end{defin}

Note that a periodic quasi-isomorphism is in  fact a quasi-isomorphism, given that it induces an isomorphism on the only two non-zero cohomologies.

\begin{remark}\label{rem.yoneda-eq}
Let $\upvartheta, \upvartheta'$ denote the Ext classes associated to the periodic projective resolutions $\PP, \PP'$ respectively.
According to the theory of $n$-fold extensions of Yoneda, a  periodic quasi-isomorphism $g \colon \PP \to \PP'$ gives an equivalence between the two extensions, and hence we have that $\upvartheta= \upvartheta'$ in $\Ext^n_{\Upgamma}(M,M)$.
\end{remark}
	
The converse is also true.
	
\begin{prop}\label{prop.yoneda-ext}
Two periodic projective resolutions of $M$ of the same length have the same extension class if and only if there exists a periodic quasi-isomorphism between them.
\end{prop}
\begin{proof}
Let $\upvartheta  \in \Ext^n_{\Upgamma}(M,M)$ be the class of the periodic projective  resolution 
\begin{equation}\label{eq.periodic-proj}
0 \to M \xrightarrow{\upbeta}  P_{n-1} \xrightarrow{d_{n-1}}P_{n-2} \to \dots \to P_{1} \xrightarrow{d_{1}} P_0 \xrightarrow{\upalpha} M  \to 0.
\end{equation}		
Following \cite{Yo}, there is a natural map from the set of $n$-fold extensions $E^n(M,M)$ to $\Ext^n_{\Upgamma}(M,M)$. To describe this, first choose a projective resolution of $M$, truncate it at length $n$ to obtain
\[
0 \to \Ker(\updelta_{n-1})\xrightarrow{\upiota}  E_{n-1} \xrightarrow{\updelta_{n-1}}E_{n-2} \to \dots \to E_{1} \xrightarrow{\updelta_{1}} E_0 \xrightarrow{\upvarepsilon} M  \to 0,
\]
then choose a chain map $f$ that lifts the identity on $M$, as follows
\[
\begin{tikzpicture}[xscale=0.75,yscale=1.5]
\node (Am1) at (-6.25,0) {$0$};
\node (A0) at (-4.5,0) {$\Ker(\updelta_{n-1})$};
\node (A1) at (-1.75,0) {$E_{n-1}$};
\node (A2) at (0.75,0) {$E_{n-2}$};
\node (A3) at (3,0) {$\hdots$};
\node (A4) at (5,0) {$E_1$};
\node (A5) at (7,0) {$E_0$};
\node (A6) at (9,0) {$M$};
\node (A7) at (10.5,0) {$0$};
\draw[->] (Am1)--(A0);
\draw[->] (A0)--node[above]{$\scriptstyle\upiota$}(A1);
\draw[->] (A1)--(A2);
\draw[->] (A2)--(A3);
\draw[->] (A3)--(A4);
\draw[->] (A4)--(A5);
\draw[->] (A5)--node[above]{$\scriptstyle\upvarepsilon$}(A6);
\draw[->] (A6)--(A7);
\node (Bm1) at (-6.25,-1) {$0$};
\node (B0) at (-4.5,-1) {$M$};
\node (B1) at (-1.75,-1) {$P_{n-1}$};
\node (B2) at (0.75,-1) {$P_{n-2}$};
\node (B3) at (3,-1) {$\hdots$};
\node (B4) at (5,-1) {$P_1$};
\node (B5) at (7,-1) {$P_0$};
\node (B6) at (9,-1) {$M$};
\node (B7) at (10.5,-1) {$0$};
\draw[->] (Bm1)--(B0);
\draw[->] (B0)--node[above]{$\scriptstyle\upbeta$}(B1);
\draw[->] (B1)--(B2);
\draw[->] (B2)--(B3);
\draw[->] (B3)--(B4);
\draw[->] (B4)--(B5);
\draw[->] (B5)--node[above]{$\scriptstyle\upalpha$}(B6);
\draw[->] (B6)--(B7);
\draw[->]  (A0) --node[right] {$\scriptstyle f_{n}$}(B0);
\draw[->] (A1) --node[right] {$\scriptstyle f_{n-1}$}(B1);
\draw[->] (A2) -- node[right] {$\scriptstyle f_{n-2}$}(B2);
\draw[->] (A4) -- node[right] {$\scriptstyle f_{1}$} (B4);
\draw[->] (A5) -- node[right] {$\scriptstyle f_{0}$} (B5);
\draw[double distance=2pt] (A6) --(B6);
\end{tikzpicture}
\]
Such a chain map always exists and it is unique up to homotopy (see e.g.\ \ref{rem.uniqueness}). Then the Yoneda map \cite{Yo} is 
\begin{equation}\label{eq.chi}
\upchi \colon E^n(M,M) \to \tfrac{\Hom_{\Upgamma}(\Ker(\updelta_{n-1}) ,M)}{\circ\upiota (\Hom_{\Upgamma}(E_{n-1},M))} \cong \Ext^n_{\Upgamma}(M,M), 
\end{equation}
defined by sending $[0 \to M \to P_{n-1} \to \hdots \to P_0 \to M] \mapsto [f_{n}]$.
The map $\upchi$ is an isomorphism between $n$-fold extensions modulo equivalence and $ \Ext^n_{\Upgamma}(M,M).$
		
In our case it is natural to choose as a truncated projective resolution the exact sequence \eqref{eq.periodic-proj}, and the chain map $f$ to be the identity. Therefore the class $\upvartheta \in \Ext^n_{\Upgamma}(M,M)$ corresponds to the class of the identity map  in $\Hom_{\Upgamma}(M,M)/{\circ\upbeta(\Hom_{\Upgamma}(P_{n-1}, M))} $.
		
Consider now another periodic projective  resolution of length $n$ of $M$, namely
\[
0 \to M \xrightarrow{\upbeta'}  P'_{n-1} \xrightarrow{d'_{n-1}} P'_{n-2} \to \dots \to P'_{1} \xrightarrow{d'_{1}} P'_0 \xrightarrow{\upalpha'} M  \to 0,
\]
with extension class $\upvartheta'$.  There exists a map $g$ of complexes
\[
\begin{tikzpicture}[xscale=0.75,yscale=1.5]
\node (Am1) at (-5.5,0) {$0$};
\node (A0) at (-4,0) {$M$};
\node (A1) at (-1.75,0) {$P_{n-1}$};
\node (A2) at (0.75,0) {$P_{n-2}$};
\node (A3) at (3,0) {$\hdots$};
\node (A4) at (5,0) {$P_1$};
\node (A5) at (7,0) {$P_0$};
\node (A6) at (9,0) {$M$};
\node (A7) at (10.5,0) {$0$};
\draw[->] (Am1)--(A0);
\draw[->] (A0)--node[above]{$\scriptstyle\upbeta$}(A1);
\draw[->] (A1)--(A2);
\draw[->] (A2)--(A3);
\draw[->] (A3)--(A4);
\draw[->] (A4)--(A5);
\draw[->] (A5)--node[above]{$\scriptstyle\upalpha$}(A6);
\draw[->] (A6)--(A7);
\node (Bm1) at (-5.5,-1) {$0$};
\node (B0) at (-4,-1) {$M$};
\node (B1) at (-1.75,-1) {$P'_{n-1}$};
\node (B2) at (0.75,-1) {$P'_{n-2}$};
\node (B3) at (3,-1) {$\hdots$};
\node (B4) at (5,-1) {$P'_1$};
\node (B5) at (7,-1) {$P'_0$};
\node (B6) at (9,-1) {$M$};
\node (B7) at (10.5,-1) {$0$};
\draw[->] (Bm1)--(B0);
\draw[->] (B0)--node[above]{$\scriptstyle\upbeta'$}(B1);
\draw[->] (B1)--(B2);
\draw[->] (B2)--(B3);
\draw[->] (B3)--(B4);
\draw[->] (B4)--(B5);
\draw[->] (B5)--node[above]{$\scriptstyle\upalpha'$}(B6);
\draw[->] (B6)--(B7);
\draw[->]  (A0) --node[right] {$\scriptstyle g_{n}$}(B0);
\draw[->] (A1) --node[right] {$\scriptstyle g_{n-1}$}(B1);
\draw[->] (A2) -- node[right] {$\scriptstyle g_{n-2}$}(B2);
\draw[->] (A4) -- node[right] {$\scriptstyle g_{1}$} (B4);
\draw[->] (A5) -- node[right] {$\scriptstyle g_{0}$} (B5);
\draw[double distance=2pt] (A6) --(B6);
\end{tikzpicture}
\]
such that $[g_{n}] = \upvartheta'$.  Then by \cite[3.5]{Yo}
\[ 
\upvartheta = \upvartheta' \iff [g_{n}] = [\Id_M ] \iff g_{n} - \Id_M \in \image (\circ\upbeta).
\]		
Now, suppose that $g_{n} - \Id_M \in \image (\circ \upbeta)$, so there exists $h \in \Hom_{\Upgamma}(P_{n-1}, M)$ such that $g_{n} - \Id_M  = h \circ \upbeta$.  Then consider the map $\mathsf{g} \colon \PP \to \PP' $ defined by
\[ 
\mathsf{g_i} = 
\begin{cases}
g_i &\mbox{ if } 0\leq i\leq n-2,\\
g_{n-1} - \upbeta' \circ h&\mbox{ if }i=n-1.
\end{cases}
\]
Since
\[ 
\mathsf{g}_{n-1} \circ \upbeta = {g_{n-1}} \circ \upbeta - \upbeta' \circ h \circ \upbeta = \upbeta' \circ g_{n} -\upbeta' \circ g_{n} + \upbeta' = \upbeta',
\]
$\mathsf{g}$ is a periodic quasi-isomorphism. The converse is clear, by \ref{rem.yoneda-eq}.
\end{proof}
	
\begin{lemma}\label{lem.surj-map-hom-ext} 
A periodic projective resolution $\PP$ of $M$ determines a surjective map 
\[
\uppi \colon \Hom_{\Upgamma}(M,M) \twoheadrightarrow \Ext^n_{\Upgamma}(M,M),
\] 
given by the Yoneda product $(-)\smile \upvartheta$, where $\upvartheta$ is the extension class of the periodic projective resolution. In particular, if $\upvartheta=0$ then $\Ext^n_{\Upgamma}(M,M)=0$.
\end{lemma}
	
\begin{proof}
There is a natural surjection $\Hom_\Upgamma(M,M)\twoheadrightarrow \Hom_{\Upgamma}(M,M)/{\circ\upbeta(\Hom_{\Upgamma}(P_{n-1}, M))}$. The claim is, once we identify the right hand side with $E^n(M,M)$ modulo equivalence as in the proof of \ref{prop.yoneda-ext}, that this map is given by the Yoneda product with $\upvartheta$.

This follows directly from the definition of the Yoneda product $g \smile \upvartheta$  for $g \in \Hom_{\Upgamma}(M,M)$, which we recall is defined as the extension class of the second exact sequence in the diagram below:
		\[
\begin{tikzpicture}[xscale=0.75,yscale=1.5]
\node (Am1) at (-6.25,0) {$0$};
\node (A0) at (-4.75,0) {$M$};
\node (A1) at (-2.5,0) {$P_{n-1}$};
\node (A2) at (0.75,0) {$P_{n-2}$};
\node (A3) at (3,0) {$\hdots$};
\node (A4) at (5,0) {$P_1$};
\node (A5) at (7,0) {$P_0$};
\node (A6) at (9,0) {$M$};
\node (A7) at (10.5,0) {$0$};
\draw[->] (Am1)--(A0);
\draw[->] (A0)--node[above]{$\scriptstyle\upbeta$}(A1);
\draw[->] (A1)--node[above]{$\scriptstyle d_{n-1}$}(A2);
\draw[->] (A2)--(A3);
\draw[->] (A3)--(A4);
\draw[->] (A4)--(A5);
\draw[->] (A5)--node[above]{$\scriptstyle\upalpha$}(A6);
\draw[->] (A6)--(A7);
\node (Bm1) at (-6.25,-1) {$0$};
\node (B0) at (-4.75,-1) {$M$};
\node (B1) at (-2.5,-1) {$M\oplus_g P_{n-1}$};
\node (B2) at (0.75,-1) {$P_{n-2}$};
\node (B3) at (3,-1) {$\hdots$};
\node (B4) at (5,-1) {$P_1$};
\node (B5) at (7,-1) {$P_0$};
\node (B6) at (9,-1) {$M$};
\node (B7) at (10.5,-1) {$0$};
\draw[->] (Bm1)--(B0);
\draw[->] (B0)--node[above]{$\scriptstyle$}(B1);
\draw[->] (B1)--node[above]{$\scriptstyle (0,d_{n-1})$}(B2);
\draw[->] (B2)--(B3);
\draw[->] (B3)--(B4);
\draw[->] (B4)--(B5);
\draw[->] (B5)--node[above]{$\scriptstyle\upalpha'$}(B6);
\draw[->] (B6)--(B7);
\draw[->]  (A0) --node[right] {$\scriptstyle g$}(B0);
\draw[->] (A1) --node[right] {$\scriptstyle $}(B1);
\draw[double distance=2pt] (A2)--(B2);
\draw[double distance=2pt] (A4)--(B4);
\draw[double distance=2pt] (A5)--(B5);
\draw[double distance=2pt] (A6) --(B6);
\end{tikzpicture}
\]
where $M \oplus_g P_{n-1}$ is the pushout along $g$ and $\upbeta$, namely
\[ 
M \oplus P_{n-1}/{\{ g(n)- \upbeta (n) \ | \ n \in M\}}.
\] 
Since $g \smile \upvartheta$ is the second sequence, it follows that $\upchi$ sends $g \smile \upvartheta$ to $[g]$, and thus
\[
\begin{tikzpicture}
\node (A) at (0,0) {$\Hom_\Upgamma(M,M)$};
\node (B) at (4,0) {$\Hom_\Upgamma(M,M)/\Im(\circ\upbeta)$};
\node (C) at (4,-1.5) {$E^n(M,M)/\sim$};
\draw[->>] (A) --(B);
\draw[->] (A) --node[below left]{$\scriptstyle (-)\smile\upvartheta$}(C);
\draw[->] (C) --node[right]{$\scriptstyle \upchi$}(B);
\end{tikzpicture}
\]
commutes, as claimed.

Now $\upvartheta =0$ in $\Ext^n_{\Upgamma}(M,M)$ if and only if $\Id_M$ belongs to $\Im(\circ\upbeta)$, which is if and only if $\Id_M = h \circ \upbeta $ for some $h \colon P_{n-1}\to  M$.  In that case, for any $g \in \Hom_{\Upgamma}(M,M)$, $g= g \circ \Id_M = g \circ h \circ \upbeta$, so that $g\in\Im(\circ\upbeta)$ and hence it is zero in the $\Ext$ group. Since $\uppi$ is a surjection, this implies that $\Ext_{\Upgamma}^n(M,M)=0$.
\end{proof}

The following special case will be used in future sections.

\begin{cor}\label{cor.simple-module-lambda}
Let $\PP, \PP'$ be two periodic projective resolutions of length $n$ of a simple $\Upgamma$-module $S$. Then for the associated extension classes $\upvartheta, \upvartheta' \in \Ext_{\Upgamma}^n(S,S)$, either $\upvartheta = \upvartheta'=0$ or $\upvartheta'=\uplambda \upvartheta$ for some $\uplambda\in\mathbb{C}^{\times}$.
\end{cor}
	
\begin{proof}
By \ref{lem.surj-map-hom-ext}, there is a surjection $\mathbb{C}\cong \Hom_{\Upgamma}(S,S)\twoheadrightarrow \Ext_{\Upgamma}^n(S,S)$.  It follows that $\Ext_{\Upgamma}^n(S,S)$ is either $\mathbb{C}$ or zero.  Write $\upvartheta , \upvartheta' \in \Ext^n_{\Upgamma}(S,S) $ for the extension class associated to $\PP$ and $\PP'$ respectively.  If $\Ext_{\Upgamma}^n(S,S)=0$, then $\upvartheta=\upvartheta'=0$.
The other case is when $\Ext_{\Upgamma}^n(S,S)\cong \mathbb{C}$, in which case $\upvartheta' =\uplambda \upvartheta$ for some $\uplambda \in \mathbb{C}$. If  $\uplambda =0$, then $\upvartheta'=0$, so by \ref{lem.surj-map-hom-ext} necessarily $\Ext_{\Upgamma}^n(S,S)=0$, which is a contradiction.  Thus $\upvartheta' = \uplambda \upvartheta$ with $\uplambda \in \mathbb{C}^{\times}$.
	\end{proof}

\section{Strictly unital  \texorpdfstring{$A_\infty$}{Ainfty}-algebras}\label{sec:StrictlyUnital}

In this section, we recall some definitions and technical results about strictly unital $A_\infty$-algebras, mainly following \cite{LH}.

\begin{defin}
An $A_\infty$-algebra $(\scrA, \{m_i\}_{i\geq 1})$ is \emph{strictly unital} if it is equipped with an element of degree zero $\upeta_\scrA$ such that $m_2(a\otimes\, \upeta_\scrA)= m_2(\upeta_\scrA\otimes a)= a$ for all $a \in \scrA$, and for all $i \neq 2$,
\[ 
m_i(a_1\otimes \hdots \otimes a_{j-1} \otimes\, \upeta_\scrA\otimes a_{j+1} \otimes \hdots \otimes a_i)=0
\]
for all $j \in \{1, \hdots, i\}$ and all $a_1, \hdots, a_{j-1}, a_{j+1}, \hdots,  a_i \in \scrA$.
\end{defin}
A $\DG$-algebra with unity, considered as an $A_\infty$-algebra, is strictly unital.

Given an $A_\infty$-morphism $f \colon \scrA \rightsquigarrow \scrB$, we write $f_i \colon \scrA^{\otimes i} \to \scrB$ for its components, which are of degree $1-i$.

\begin{defin}
Let $\scrA, \scrB$ be two strictly unital $A_\infty$-algebras. An $A_\infty$-morphism $f \colon \scrA \rightsquigarrow \scrB$ is \emph{strictly unital} if $f_1 (\upeta_\scrA) = \upeta_\scrB$ and for every $i \neq 1$,
\[ 
f_i(a_1\otimes \hdots \otimes a_{j-1} \otimes\, \upeta_\scrA\otimes a_{j+1} \otimes \hdots \otimes a_i)=0
\]
for all $j \in \{1, \hdots, i\}$ and all $a_1, \hdots, a_{j-1}, a_{j+1}, \hdots,  a_i \in \scrA$.
\end{defin}

It is clear that composition of strictly unital $A_\infty$-morphisms is again a strictly unital $A_\infty$-morphism. Trivially, any unit-preserving $\DG$-algebra morphism of unital $\DG$-algebras is a strictly unital $A_\infty$-morphism.

Kadeishvili's Theorem, known more generally as the Homotopy Transfer Theorem, equips the cohomology of a $\DG$-algebra with an $A_\infty$-structure, which is unique up to quasi-isomorphism \cite{Ka}. We are interested here in the strictly unital version of this theorem, or to be precise the following strictly unital version of the improved result due to Markl \cite{Markl}.

\begin{prop}[{\cite[3.2.4.1, 3.2.4.2]{LH}}]\label{prop.strictly-unital-HT}
	Any strictly unital $A_\infty$-algebra $\scrA$ satisfying $\HH^*(\scrA)\neq 0$ admits a strictly unital minimal model $\scrA'$ such that the $A_\infty$-quasi-isomorphisms 
	\[ 
	i \colon \scrA' \rightsquigarrow \scrA, \quad p \colon \scrA \rightsquigarrow \scrA'
	\] are strictly unital.
\end{prop}

This implies that given a unital $\DG$-algebra $\scrA$ with non-trivial cohomology, its cohomology $\HH^*(\scrA)$ can be equipped with a strictly unital $A_\infty$-structure $(\HH^*(\scrA), \{m_i\}_{i \geq 2})$ such that the extensions to $A_\infty$-morphisms of the inclusion $\HH^*(\scrA)\to \scrA$ and projection $\scrA \to \HH^*(\scrA)$  (note that the inclusion and projection maps are not unique),  are \emph{both} strictly unital $A_\infty$-quasi-isomorphisms.

\begin{assum}\label{assum:nonzero}
In the following, we will always assume that $\DG$-algebras $\scrA$ are unital and that the identity is not a coboundary, hence $\mathsf{1}\colonequals [\Id_\scrA] \in \HH^0(\scrA)$ is non-zero.
\end{assum} 
For a unital $\DG$-algebra (or for a strictly unital $A_\infty$-algebra), the identity is easily seen to be non-trivial in cohomology if and only if the cohomology is non-trivial. 
This means that we can apply \ref{prop.strictly-unital-HT} to any such $\DG$-algebra.

\begin{lemma}\label{lem.connective-truncation}
Given a unital $\DG$-algebra $\scrA$ satisfying \textnormal{\ref{assum:nonzero}}, the induced strictly unital $A_\infty$-structure on its cohomology restricts to an $A_\infty$-structure on $\mathbb{C}\mathsf{1} \oplus \HH^{>0}(\scrA)$.
\end{lemma}
\begin{proof}
Let $\scrA$ be a unital $\DG$-algebra, and let $i \colon \HH^*(\scrA) \to \scrA$ be a morphism of complexes inducing the identity in cohomology, and such that  $i (\mathsf{1})= \Id_\scrA$. Let $p \colon \scrA \to \HH^*(\scrA)$ be a unit-preserving (so $p(\Id_\scrA)=\mathsf{1}$) morphism of complexes such that $p \circ i = \Id$. By \ref{prop.strictly-unital-HT}, there exists a strictly unital $A_\infty$-structure $(\HH^*(\scrA), \{m_i\}_{i \geq 2})$ such that $i, p$ can be extended to strictly unital $A_\infty$-quasi-isomorphisms. 

We claim that for degree reasons, this $A_\infty$-structure restricts to an $A_\infty$-structure on $\mathbb{C}\mathsf{1} \oplus \HH^{>0}(\scrA)$. First, recall that $m_2(a_1,a_2)$ lives in degree $|a_1| + |a_2|$, which for inputs in $\mathbb{C}\mathsf{1} \oplus \HH^{>0}(\scrA)$ is always strictly greater than zero unless $|a_1|=|a_2|=0$. But in that case $a_1= \uplambda \mathsf{1}, a_2 = \upmu \mathsf{1}$ for some $\uplambda, \upmu \in \mathbb{C}$, so $m_2(\uplambda \mathsf{1}, \upmu \mathsf{1})= \uplambda \upmu \mathsf{1}\in \mathbb{C}\mathsf{1} \oplus \HH^{>0}(\scrA)$. Thus always $m_2$ restricts to $\mathbb{C}\mathsf{1} \oplus \HH^{>0}(\scrA)$.   For the higher products, for inputs in $\mathbb{C}\mathsf{1} \oplus \HH^{>0}(\scrA)$ by strict unitality  $m_i(a_1, \hdots , a_i)$ is non-zero only if $|a_j| \geq 1$ for all $1 \leq j \leq i$, and in that case  
	\[
	|m_i(a_1, \hdots , a_i)| = \sum_{j=1}^{i} |a_j| +2 - i \geq i +2 - i =2.\qedhere
	\]
\end{proof}
\begin{defin}\label{definition nonneg subalg}
Given a unital $\DG$-algebra $\scrB$  satisfying \textnormal{\ref{assum:nonzero}}, we call the $A_\infty$-algebra $\scrN_\scrB=\mathbb{C}\mathsf{1} \oplus \HH^{>0}(\scrB)$ equipped with the $A_\infty$-structure given by \ref{lem.connective-truncation}, the \emph{unitally positive $A_\infty$-algebra} associated to $\scrB$.
\end{defin}

We next show that this new object is well-defined.

\begin{prop}\label{scrN is well-defined}
If two unital $\DG$-algebras $\scrA$, $\scrB$ satisfying \textnormal{\ref{assum:nonzero}} are quasi-isomorphic as unital $\DG$-algebras, then there exists an $A_\infty$-quasi-isomorphism between $\scrN_\scrA$ and $\scrN_\scrB$.
\end{prop}
\begin{proof}
	Let $\scrA$ and $\scrB$ be two unital $\DG$-algebras which are quasi-isomorphic as unital $\DG$-algebras, i.e., such that they are connected by a zigzag of quasi-isomorphisms of unital $\DG$-algebras which are unit-preserving. It is well known that one can replace a zigzag of $\DG$-algebra quasi-isomorphisms with an $A_\infty$-quasi-isomorphism $\scrA \rightsquigarrow \scrB$, see e.g. \cite[11.4.9]{LV}, but we could not find a strictly unital version of this result, so we prove it for completeness. 
	Let \[\scrA=\scrA_0 \xleftarrow{f_1} \scrA_1 \xrightarrow{f_2} \scrA_2 \xleftarrow{} \hdots \xrightarrow{} \scrA_n=\scrB\]
	be a zigzag of quasi-isomorphisms of $\DG$-algebras, such that for every $j$, $\scrA_j$ is unital and $f_j$ is unit-preserving. Without loss of generality, we show that we can replace $f_1$ with a strictly unital $A_\infty$-quasi-isomorphism in the opposite direction.
	Denote by $\HH^*(\scrA_i)$ the strictly unital minimal model of $\scrA_i$ obtained by \ref{prop.strictly-unital-HT}, and by $i\colon \HH^*(\scrA_j) \rightsquigarrow \scrA_j$, $p \colon \scrA_j \rightsquigarrow \HH^*(\scrA_j)$ the $A_\infty$-quasi-isomorphisms. The composition
\[
p \circ f_1 \circ i \colon \HH^*(\scrA_1) \rightsquigarrow \scrA_1 \xrightarrow{f_1} \scrA \rightsquigarrow \HH^*(\scrA)
\]
is a strictly unital $A_\infty$-quasi-isomorphism between minimal $A_\infty$-algebras, so it is a  strictly unital $A_\infty$-isomorphism. Then by \cite[3.2.4.6]{LH} its $A_\infty$-inverse $\upbeta \colon \HH^*(\scrA) \rightsquigarrow \HH^*(\scrA_1)$ is also strictly unital, and the composition
\[ i \circ \upbeta \circ p \colon \scrA \rightsquigarrow \HH^*(\scrA) \rightsquigarrow \HH^*(\scrA_1) \rightsquigarrow \scrA_1
\]
is a strictly unital $A_\infty$-quasi-isomorphism. Proceeding in this way for all the maps in the zigzag which point left, we obtain a strictly unital $A_\infty$-quasi-isomorphism $f \colon \scrA \rightsquigarrow \scrB$.

We next claim that the strict unitality implies that the composition
$p \circ f \circ i \colon \HH^*(\scrA) \rightsquigarrow \scrA \rightsquigarrow \scrB \rightsquigarrow \HH^*(\scrB)$, which is a strictly unital $A_\infty$-isomorphism, 
restricts to an $A_\infty$-isomorphism between  $\mathbb{C}\mathsf{1}\oplus \HH^{>0}(\scrA )$ and $\mathbb{C}\mathsf{1}\oplus \HH^{>0}(\scrB )$. In fact, for every $n>1$, if some $x_i$ is of degree zero then $f_n(x_1,\hdots ,x_n)=0$. Else, when applied to inputs in $\HH^{>0}(\scrA)$,
\[
 | f_n(x_1,\hdots ,x_n)| = \sum_{i=1}^{n} |x_i| +1 -n \geq n+1-n=1.
 \]
The linear component $f_1$ is of degree zero, so we just need to check that $f_1(\mathsf{1})=\mathsf{1}$, which follows by strict unitality.  Combining, this proves the claim.  Thus since $f_1$ (which has degree zero) is still an isomorphism after restriction, the result follows.
\end{proof}

\begin{remark}\label{rem.adjoint-to-incl}
The construction of $\scrN$ can be seen to define a right adjoint to the inclusion $ (\rm{HoDGA})^{\geq 1} \hookrightarrow \rm{HoDGA}$ of the full subcategory of DG-algebras $\scrA$ such that $\HH^0(\scrA)=\mathbb{C}$ and  $\HH^{< 0}(\scrA)=0$ into the homotopy category of DG-algebras satisfying \textnormal{\ref{assum:nonzero}}.

An analogue result, at the level of $A_\infty$-algebras (i.e., not on the homotopy category), is shown in \cite[6.3]{Rodriguez Rasmussen}, while the last part of the proof of \ref{scrN is well-defined} follows from 6.2 of \emph{loc.\ cit}.
\end{remark}

\section{The Case of Idempotents}\label{sec:idems}

In this section consider the setting where $\rmA$ is a $\mathbb{C}$-algebra with idempotent $e\in\rmA$.  Set $\Acon\colonequals \rmA/\rmA e\rmA$, then the ring homomorphism $\rmA\to\Acon$ induces the restriction and extension of scalars adjunction
\[
\begin{tikzpicture}[xscale=1]
\node (d1) at (0,0) {$\mod\Acon$};
\node (e1) at (3,0) {$\mod \rmA._{\phantom{con}}$};
\draw[->,transform canvas={yshift=+0.4ex}] (d1) to node[above] {$\scriptstyle i_*$} (e1);
\draw[<-,transform canvas={yshift=-0.4ex}] (d1) to node [below]  {$\scriptstyle i^* $} (e1);
\end{tikzpicture}
\]

\begin{setup}\label{Acon setup}
Given a finite-dimensional simple module $S\in\mod\rmA$, suppose that there exists a complex of finitely generated projective $\rmA$-modules
\[
\scrQ\colonequals\quad  0 \rightarrow Q_{n-1} \rightarrow Q_{n-2} \rightarrow \dots \rightarrow Q_{1} \rightarrow Q_0 \rightarrow 0
\]
which gives a $\rmA$-projective resolution of $S$, such that $\PP\colonequals i^*\scrQ $ satisfies
\[
\mathrm{H}^{-j}(\PP) \cong \left\{ \begin{array}{ll}
 i^*S & \text{ if $j\in \{ 0, n-1\}$ } \\
 0 & \text{ otherwise}.
\end{array}
\right.
\]
As such, we may choose $\upalpha$ and $\upbeta$ such that
\[
0\to i^*S\xrightarrow{\upbeta} P_{n-1}\to P_{n-2}\to\hdots\to P_1\to P_0\xrightarrow{\upalpha} i^*S
\to 0
\]
is exact, and thus induce a length $n$ periodic projective resolution of $i^*S$. Using Setup~\ref{key setup}, and 
\ref{Definition: TEA}, we thus obtain a trivial extension $\DG$-algebra $\scrT$, and the quasi-isomorphic $\DG$-algebra $\scrEnd_{\Acon}(\PP)$.
\end{setup}

\begin{notation}\label{not:abuse1}
We will often abuse notation and simply write $S$ for $i^*(S)$, that is, $S$ viewed as an $\Acon$-module.  This is justified, since the subscript on the Ext groups below will always tell us in which category $S$ should be viewed.
\end{notation}

\begin{remark}\label{rem:allExtsfd}
Since $\Hom_{\Acon}(\Acon,S)\cong S$, and in Setup~\ref{Acon setup} $S$ is assumed to be finite dimensional, it follows that $\Hom_{\Acon}(P,S)$ is finite dimensional for all finitely generated projective $\Acon$-modules $P$.  In particular, the following statements hold.
\begin{enumerate}
\item In the notation of \ref{lem.map-pi} and \ref{eq.coh-t}, necessarily $\Im(\circ d_0)\subseteq\Hom_{\Acon}(P_0,S)$ is finite dimensional.
\item For $i=1,\hdots,n-1$, each $\Ext^i_{\Acon}(S,S)$ is finite dimensional, given that it can be computed as a quotient of a subspace of $\Hom_{\Acon}(P_i,S)$.
\end{enumerate}
\end{remark}

\subsection{Uniqueness for simples}
It is well-known, see e.g.\ \cite[4.4]{ST}, that given another $\rmA$-projective resolution $\scrQ'$ of $S$, the $\DG$-algebras $\scrEnd_\rmA(\scrQ)$ and $\scrEnd_\rmA(\scrQ')$ are quasi-isomorphic. However, the situation we will be interested in is the one of two rings $\rmA,\mathrm{B}$ with idempotents such that $\Acon \cong \Bcon$. Thus, we will need to compare two periodic projective resolutions of length $n$ of $S$ coming respectively from an $\rmA$-projective resolution $\scrQ$ of length $n$ of $S$ and an $\mathrm{B}$-projective 
resolution $\scrQ'$ of length $n$ of $S$. 

This turns out to be possible using the fact that the module is simple, via \ref{cor.simple-module-lambda}. The following works more generally.

\begin{lemma}\label{lem.qisom-endomo-simple}
Let $\Upgamma$ be a $\mathbb{C}$-algebra, $S$ a simple $\Upgamma$-module, and let $\PP, \PP'$ be two periodic projective resolutions of length $n$ of $S$, as in \eqref{eq:Yon1}.
Then there exists a quasi-isomorphism between $\PP$ and $\PP'$, so  $\scrEnd_{\Upgamma}(\PP)$ and $\scrEnd_{\Upgamma}(\PP')$ are quasi-isomorphic as unital $\DG$-algebras. 
\end{lemma}
\begin{proof}
Writing $\upvartheta, \upvartheta' \in \Ext_{\Upgamma}^n(S,S)$ for the extension classes of $\PP$ and $\PP'$, then by \ref{cor.simple-module-lambda} either $\upvartheta=\upvartheta'=0$ or $\upvartheta' =\uplambda \upvartheta$ for some $\uplambda \in \mathbb{C}^\times$.

If $\upvartheta=\upvartheta'=0$ then by \ref{prop.yoneda-ext} $\PP$ and $\PP'$ are linked by a periodic quasi-isomorphism, so that by \cite[4.4]{ST} the $\DG$-algebras $\scrEnd_{\Upgamma}(\PP)$ and $\scrEnd_{\Upgamma}(\PP')$ are quasi-isomorphic. Notice that the proof of \cite[4.4]{ST} works by constructing a roof of unit-preserving quasi-isomorphisms of unital $\DG$-algebras $\scrEnd_{\Upgamma}(\PP) \leftarrow N \to \scrEnd_{\Upgamma}(\PP')$, so that $\scrEnd_{\Upgamma}(\PP)$ and $\scrEnd_{\Upgamma}(\PP')$ are quasi-isomorphic as unital $\DG$-algebras.

If $\upvartheta' = \uplambda \upvartheta$ with $\uplambda \in \mathbb{C}^{\times}$, there exists a chain map $g$
\[
\begin{tikzpicture}[xscale=0.75,yscale=1.5]
\node (Am1) at (-5.5,0) {$0$};
\node (A0) at (-4,0) {$S$};
\node (A1) at (-1.75,0) {$P_{n-1}$};
\node (A2) at (0.75,0) {$P_{n-2}$};
\node (A3) at (3,0) {$\hdots$};
\node (A4) at (5,0) {$P_1$};
\node (A5) at (7,0) {$P_0$};
\node (A6) at (9,0) {$S$};
\node (A7) at (10.5,0) {$0$};
\draw[->] (Am1)--(A0);
\draw[->] (A0)--node[above]{$\scriptstyle\upbeta$}(A1);
\draw[->] (A1)--node[above]{$\scriptstyle d_{n-1}$}(A2);
\draw[->] (A2)--(A3);
\draw[->] (A3)--(A4);
\draw[->] (A4)--(A5);
\draw[->] (A5)--node[above]{$\scriptstyle\upalpha$}(A6);
\draw[->] (A6)--(A7);
\node (Bm1) at (-5.5,-1) {$0$};
\node (B0) at (-4,-1) {$S$};
\node (B1) at (-1.75,-1) {$P'_{n-1}$};
\node (B2) at (0.75,-1) {$P'_{n-2}$};
\node (B3) at (3,-1) {$\hdots$};
\node (B4) at (5,-1) {$P'_1$};
\node (B5) at (7,-1) {$P'_0$};
\node (B6) at (9,-1) {$S$};
\node (B7) at (10.5,-1) {$0$};
\draw[->] (Bm1)--(B0);
\draw[->] (B0)--node[above]{$\scriptstyle\upbeta'$}(B1);
\draw[->] (B1)--(B2);
\draw[->] (B2)--(B3);
\draw[->] (B3)--(B4);
\draw[->] (B4)--(B5);
\draw[->] (B5)--node[above]{$\scriptstyle\upalpha'$}(B6);
\draw[->] (B6)--(B7);
\draw[->]  (A0) --node[right] {$\scriptstyle g_{n}$}(B0);
\draw[->] (A1) --node[right] {$\scriptstyle g_{n-1}$}(B1);
\draw[->] (A2) -- node[right] {$\scriptstyle g_{n-2}$}(B2);
\draw[->] (A4) -- node[right] {$\scriptstyle g_{1}$} (B4);
\draw[->] (A5) -- node[right] {$\scriptstyle g_{0}$} (B5);
\draw[double distance=2pt] (A6) --(B6);
\end{tikzpicture}
\]
such that $g_{n} - \uplambda \Id_S = h \circ \upbeta$, for some $h \in \Hom_{\Upgamma}(P_{n-1}, S)$. As in the proof of \ref{prop.yoneda-ext}, we now modify the map $g$.  Consider $\mathsf{g}\colon\PP \to\PP'$ defined by
\[
\mathsf{g}_i = 
\begin{cases}
g_i &\mbox{if } i=0,\hdots, n-2\\
g_{n-1} - \upbeta' \circ h&\mbox{if } i=n-1.
\end{cases}
\]
It is easy to see that
\[
\mathsf{g}_{n-1} \circ \upbeta = {g_{n-1}}\circ \upbeta  - \upbeta' \circ h \circ \upbeta = \upbeta' \circ g_{-n} - \upbeta' \circ h \circ \upbeta = \uplambda \upbeta' .  
\]
From this, $\mathsf{g}$ is easily seen to be a quasi-isomorphism and therefore the  $\DG$-algebras of endomorphisms are again quasi-isomorphic by \cite[4.4]{ST}.
\end{proof}

\begin{cor}\label{scrN is well-defined 2}
Let $\Upgamma$ be a $\mathbb{C}$-algebra, $S\in\mod\Upgamma$ be simple, and let $\PP,\PP'$ be two periodic projective resolutions of the simple module $S$, of the same length.  Then the strictly unital $A_\infty$-algebras $ \mathbb{C}\mathsf{1} \oplus \HH^{>0}(\scrEnd_{\Upgamma}(\PP))$ and $ \mathbb{C}\mathsf{1} \oplus \HH^{>0}(\scrEnd_{\Upgamma}(\PP'))$ are quasi-isomorphic.
\end{cor}
\begin{proof}
Using the fact that $S$ is simple, $\scrEnd_{\Upgamma}(\PP)$ and $\scrEnd_{\Upgamma}(\PP')$ are quasi-isomorphic as unital $\DG$-algebras by \ref{lem.qisom-endomo-simple}. The conclusion then follows from the general \ref{scrN is well-defined}.
\end{proof}

\subsection{Main reconstruction result}\label{sec:main-recons}
Under Setup~\ref{Acon setup}, since $\PP\colonequals i^*\scrQ$, consider the morphism 
\[
f \in \scrEnd_{\rmA}(\scrQ)\to i^*f \in \scrEnd_{\Acon}(\PP).
\]  
It is clear that $\updelta(i^*f)=i^*\updelta(f)$, and that the above is a morphism of unital $\DG$-algebras.  

This is not a quasi-isomorphism, but it will become one once we pass to the unitally positive  $A_\infty$-algebra.  The point will then be (in \ref{cor:main_reconstruct}) that $\scrEnd_{\rmA}(\scrQ)$, which is built on the large ring $\rmA$, can be recovered from the isomorphism class of the smaller ring $\Acon$.  The following is our main technical result.

\begin{thm}\label{Main DG result}
Under Setup~\textnormal{\ref{Acon setup}}, there is an $A_\infty$-isomorphism between the strictly unital minimal model of $\scrEnd_{\rmA}(\scrQ)$ and  ${\scrN=\mathbb{C}\mathsf{1}\oplus \HH^{>0}(\scrEnd_{\Acon}(\PP))}$. 
\end{thm}
\begin{proof}
Let  $\upvarepsilon \colon \scrQ \to  S$  be the $\rmA$-projective resolution of $S$ of length $n$ in Setup \ref{Acon setup}, and define $\upalpha$ as the map
\begin{center}
\begin{tikzcd}
i^*(\scrQ)=\PP \arrow[d, "i^*(\upvarepsilon)"'] \arrow[r, "\iota"] & \scrP \arrow[ld, "\upalpha"] \\
S   ,                                                             &                             
\end{tikzcd}
\end{center}
where $\upiota \colon \PP \to \scrP$ is the inclusion as in \S\ref{general qis section}.

Consider the following diagram
\begin{equation}\label{eq.diagram}
		\begin{tikzcd}
			\scrEnd_\rmA(\scrQ) \arrow[rd, "\upvarepsilon \circ "', bend right=15] \arrow[r, "i^*"] & \scrEnd_{\Acon}(\PP) \arrow[r, "\upiota\circ "]        & {\scrHom_{\Acon}(\PP,\scrP)} \arrow[ld, "\upalpha\circ", bend left=15] \\ & {\scrHom_\rmA(\scrQ, S)\cong \scrHom_{\Acon}(\PP,S)}, &                                                                 
		\end{tikzcd}
	\end{equation}
which commutes because  $\upalpha \circ \upiota = i^*(\upvarepsilon)$. The maps $\upvarepsilon \colon \scrQ \to  S$ and $\upalpha \colon \scrP \to S$ are projective resolutions, hence the maps $\upvarepsilon\circ $ and $ \upalpha \circ $ in the diagram above are quasi-isomorphisms of complexes. In fact, the cone of $\upvarepsilon \circ $ is isomorphic to $\scrHom_\rmA(\scrQ, \operatorname{cone}(\upvarepsilon))$, which is acyclic, because any morphism of complexes from a bounded-above complex of projective modules to an acyclic complex is homotopic to the zero map (see e.g. \cite[III.5.24]{GM}), and analogously for $ \upalpha\circ $.

This implies that we can calculate the cohomology of $\scrEnd_\rmA(\scrQ) $ as a graded vector space using \eqref{eq.coh-ext}, via
\begin{align}\label{eq.coh.ext2}
\Ext^i_\rmA(S, S)&\cong \HH^i(\scrEnd_\rmA(\scrQ))   \nonumber\\
&\cong \HH^i(\scrHom_{\Acon}(\PP, \scrP))  \nonumber\\
&= \begin{cases} 0 &\quad i < 0,\ i \geq n\\ 
\Ext^i_{\Acon}(S,S) &\quad i=0,  \hdots , n-2 \\ \Ext^{n-1}_{\Acon}(S,S) \oplus \image(\circ d_0) &\quad i =n-1.
\end{cases}
\end{align}

Using the diagram  \eqref{eq.diagram} we can relate the cohomology of $\scrEnd_\rmA(\scrQ) $, i.e., $\Ext^*_\rmA(S,S)$, to the cohomology of $\scrEnd_{\Acon}(\PP)$, which was calculated in \S\ref{sec.cohomology_end}. From the commutativity of the diagram \eqref{eq.diagram} and the fact that $\upvarepsilon\circ , \upalpha\circ $ are quasi-isomorphisms, the composition along the top line of \eqref{eq.diagram} is a quasi-isomorphism. Thus $i^* \colon \scrEnd_\rmA(\scrQ) \to \scrEnd_{\Acon}(\PP)$ is injective in cohomology, and  $\upiota\circ \colon \scrEnd_{\Acon}(\PP) \to \scrHom_{\Acon}(\PP, \scrP)$ is surjective in cohomology.

Now consider the unitally positive $A_\infty$-algebra $\scrN= \mathbb{C}\mathsf{1}\oplus \HH^{>0}(\scrEnd_{\Acon}(\PP))$ defined in \ref{definition nonneg subalg}. From \ref{eq.coh-t} we have
	\begin{equation}\label{eq.B}
		\scrN^i = \begin{cases} 0 &\quad i \leq -1,\ i \geq n\\  \mathbb{C}&\quad i = 0\\
			\Ext^i_{\Acon}(S,S) &\quad i=1, \hdots , n-2 \\ \Ext^{n-1}_{\Acon}(S,S) \oplus \image(\circ d_0) &\quad i =n-1
			. \end{cases}
	\end{equation}

Consider the diagram
\[
\begin{tikzcd}
\scrEnd_{\rmA}(\scrQ) \arrow[r, "i^*"] \arrow[d, shift left] & \scrEnd_{\Acon}(\PP) \arrow[d, "p", shift left]     \\
\HH^*(\scrEnd_{\rmA}(\scrQ)) \arrow[u, "j", shift left]      & \HH^*(\scrEnd_{\Acon}(\PP)) \arrow[u, shift left] & {\scrN\colonequals \mathbb{C}\mathsf{1}\oplus \HH^{>0}(\scrEnd_{\Acon}(\PP))}. \arrow[l, hook]
\end{tikzcd}
\]
By \ref{prop.strictly-unital-HT}, the projection $p$ and the inclusion  $j$ can be extended to strictly unital $A_\infty$-quasi-isomorphisms, so we obtain a strictly unital $A_\infty$-morphism  
\[
\upvarphi\colon 	\HH^*(\scrEnd_{\rmA}(\scrQ)) \rightsquigarrow \scrEnd_{\rmA}(\scrQ) \xrightarrow{i^*} \scrEnd_{\Acon}(\PP) \rightsquigarrow \HH^*(\scrEnd_{\Acon}(\PP)), 
\]
with linear component $\upvarphi_1= p \circ i^* \circ j$.  The linear component $\upvarphi_1$ of the  $A_\infty$-morphism $\upvarphi$ can be thought of as the map induced by $i^*$ in cohomology: hence, as remarked above, it is injective. 
	
By the cohomology calculation \eqref{eq.coh.ext2} and by the same argument as in \ref{scrN is well-defined}, since $\upvarphi$ is strictly unital, a simple degree calculation shows that all its components have image contained in $\scrN$, so there is an $A_\infty$-morphism
\begin{equation}
		\upvarphi \colon \HH^*(\scrEnd_{\rmA}(\scrQ)) \rightsquigarrow \scrN= \mathbb{C}\mathsf{1}\oplus \HH^{>0}(\scrEnd_{\Acon}(\PP))\label{eqn:phidef},
\end{equation}
with linear component  $\upvarphi_1= p \circ i^* \circ j$, which is injective.

Now by \ref{rem:allExtsfd} all the vector spaces in \eqref{eq.B} are finite dimensional, and any injective linear map between vector spaces of the same dimension is an isomorphism.  Thus, now by simply comparing the cohomology of $\scrEnd_{\rmA}(\scrQ) $ in \eqref{eq.coh.ext2} with $\scrN^i$ in \eqref{eq.B}, the injective map 
\[
\upvarphi_1 \colon \HH^*(\scrEnd_{\rmA}(\scrQ)) \to \scrN= \mathbb{C}\mathsf{1}\oplus \HH^{>0}(\scrEnd_{\Acon}(\PP))
\] 
is necessarily an isomorphism, which by definition implies that $\upvarphi$ in \eqref{eqn:phidef} is an $A_\infty$-isomorphism.
\end{proof}

The following shows that under Setup~\textnormal{\ref{Acon setup}}, the strictly unital  $A_\infty$-minimal model of $\scrEnd_{\rmA}(\scrQ)$ can be reconstructed from the isomorphism class of $\Acon$.

\begin{cor}\label{cor:main_reconstruct}
Suppose that both $\rmA$ and $\mathrm{B}$ admit simples and idempotents which satisfy the conditions in Setup~\textnormal{\ref{Acon setup}}, via complexes $\scrQ_{\rmA}$ and $\scrQ_{\mathrm{B}}$ of the same length.  If $\Acon\cong\Bcon$ and both are local, then there is an $A_\infty$-quasi-isomorphism between $\scrEnd_{\rmA}(\scrQ_{\rmA})$ and $\scrEnd_{\mathrm{B}}(\scrQ_{\mathrm{B}})$.
\end{cor}
\begin{proof}
The simple $\Acon$-module $S$ is equipped with a length $n$ complex $\PP$ which builds $\scrN_S$, and by \ref{Main DG result} there is an $A_\infty$-isomorphism
\[
\HH^*(\scrEnd_{\rmA}(\scrQ_{\rmA}))\rightsquigarrow\scrN_S
\]
where the left hand side is equipped with the $A_\infty$-structure of a strictly unital minimal model. The same analysis holds for $\mathrm{B}$, where $\Bcon$ has simple $S'$ which is equipped with a complex $\PP'$ of the same length $n$ which in turn builds $\scrN_{S'}$. By \ref{Main DG result}, there is an $A_\infty$-isomorphism
\[
\HH^*(\scrEnd_{\mathrm{B}}(\scrQ_\mathrm{B}))\rightsquigarrow\scrN_{S'}.
\]
It suffices to prove that $\scrN_S$ and $\scrN_{S'}$ are $A_\infty$-isomorphic, and for this we will ultimately use \ref{scrN is well-defined}.
Indeed, since $\Acon\cong\Bcon$, there is an (exact) isomorphism of categories $F\colon \mod\Acon\to\mod\Bcon$.  Now $F$ preserves simples, so since both $\Acon$ and $\Bcon$ are local, $F S\cong S'$.  Further, $F$ takes $\PP$ to some length $n$ complex of projectives which builds a periodic projective resolution of $S'$. Clearly this isomorphism of categories induces a $\DG$-algebra isomorphism (not just quasi-isomorphism!)
\begin{equation}
F\colon\scrEnd_{\Acon}(\PP) \to\scrEnd_{\Bcon}(F\PP)\label{eqn:iso}
\end{equation}
which preserves the identity.  Now $F\PP$ is quasi-isomorphic to $\PP'$ by \ref{lem.qisom-endomo-simple}, as both give periodic resolutions of the same simple with the same length, and so again by \ref{rem.uniqueness}\eqref{rem.uniqueness3} $\scrEnd_{\Bcon}(F\PP)$ is quasi-isomorphic to $\scrEnd_{\Bcon}(\PP')$. Combining with \eqref{eqn:iso}, it follows that the unital $\DG$-algebras $\scrEnd_{\Acon}(\PP)$ and $\scrEnd_{\Bcon}(\PP')$ are quasi-isomorphic.  By \ref{scrN is well-defined 2}, $\scrN_S$ and $\scrN_{S'}$ are $A_\infty$-quasi-isomorphic, proving the result.
\end{proof}

\section{Flops}

Throughout, let $X\to\Spec R$ be a $3$-fold flopping contraction, where $X$ is smooth and $R$ is complete local.  Here, we continue to work over the complex numbers.

\subsection{Notation}\label{subsec:flopssetup}
As is well-known \cite{VdB1d}, $X$ admits a tilting bundle $\scrV=\scrO_X\oplus \scrM_1\oplus\hdots\oplus\scrM_t$, where the summands $\scrM_i$ are in bijection with the exceptional curves $C_1,\hdots,C_t$. Each $C_i\cong\mathbb{P}^1$. Writing $\rmA\colonequals \End_X(\scrV)\cong \End_R(f_*\scrV)$, then $\rmA$ is a noncommutative crepant resolution (NCCR), and there is a derived equivalence
\[
\Db(\coh X)\xrightarrow{\RHom_X(\scrV,-)}\Db(\mod \rmA).\label{eqn:NCCR}
\]
Since $R$ is complete local, the category $\mod\rmA$ has finitely many simples, which geometrically are described in \cite[3.5.7, 3.5.8]{VdB1d}.  Matching the notation in \cite{VdB1d}, we will write $S_1,\hdots,S_t$ for the simple $\rmA$-modules that are in bijection with the exceptional curves $C_1,\hdots,C_t$.  Each $S_i\cong\mathbb{C}$, and thus is finite dimensional.

Let $e$ be the idempotent of $\rmA$ corresponding to the summand $R=f_*\scrO_X$ of $f_*\scrV$.  Then by \cite{DW1} the contraction algebra $\Acon$ can be defined simply as
\[
\Acon \colonequals \rmA/\rmA e\rmA.
\]
Since $f$ is a flopping contraction, $\Acon$ is finite dimensional \cite[2.13]{DW1}, and using the same mild abuse of notation as in \ref{not:abuse1}, the simple $\Acon$-modules are $S_1,\hdots,S_t$.  

Since $\rmA$ is an NCCR, by Auslander--Buchsbaum \cite[2.3]{IR} each $S_i\in\mod \rmA$ has projective dimension three, so consider the minimal projective resolution $\scrQ^j\to S_j\to 0$.  Applying $i^*=-\otimes_{\rmA}\Acon$ gives a complex of projective $\Acon$-modules
\begin{equation}
\PP^j\colonequals \quad P^j_3\to P^j_2\to P^j_1\to P^j_0,\label{Pstosum}
\end{equation}
whose cohomology computes $\Tor_*^{\rmA}( S_j,\Acon)$.  The following is a well-known consequence of the fact that $\Acon$ is spherical.
\begin{lemma}\label{key Tor descends}
In this flops setting, for every $j=1,\hdots, t$,
\[
\Tor_k^{\rmA}( S_j,\Acon)\cong
 \left\{ \begin{array}{ll}
 \mathbb{C} & \text{ if $k\in \{ 0, 3\}$ } \\
 0 & \text{ otherwise}.
\end{array}
\right.
\]
\end{lemma}
\begin{proof}
Since $\Acon$ is a symmetric finite dimensional algebra \cite[3.3]{August}, $D\Acon\cong\Acon$ as bimodules, where $D$ is the $\mathbb{C}$-dual.  Thus
\begin{align*}
D\Tor_k^{\rmA}( S_j,\Acon)
&\cong \Ext^k_{\rmA}( S_j,D\Acon)\tag{by \cite[VI.5.1]{CE}}\\
&\cong \Ext^k_{\rmA}( S_j,\Acon)\tag{$\Acon$ is symmetric}\\
&\cong D\Ext^{3-k}_{\rmA}(\Acon, S_j)\tag{$\rmA$ is 3-CY}.\\
&\cong
 \left\{ \begin{array}{ll}
 D\,\mathbb{C} & \text{ if $k\in \{ 0, 3\}$ } \\
 0 & \text{ otherwise}.
\end{array}\tag{$\Acon$ is spherical \cite[4.7]{DW3}}
\right.
\end{align*}
The result follows, by cancelling the duality $D$.
\end{proof}

Since $i^*$ is right exact, the cokernel of the rightmost map in \eqref{Pstosum} is $ S_j$. In particular, by \ref{key Tor descends} there is an exact sequence of $\Acon$-modules
\[
0\to \Omega^4 S_j\to P^j_3\to P^j_2\to P^j_1\to P^j_0\xrightarrow{\upalpha} S_j\to 0.
\] 
Since contraction algebras are four-periodic \cite[1.1]{Dugas}, necessarily the module $\Omega^4  S_j$, which is one-dimensional by \ref{key Tor descends}, is isomorphic to $S_j$.  Passing through this isomorphism gives an exact sequence of $\Acon$-modules
\[
0\to S_j\xrightarrow{\upbeta} P^j_3\to P^j_2\to P^j_1\to P^j_0\xrightarrow{\upalpha}S_j\to 0
\] 
for each $j$ such that $1\leq j \leq t$.

We now consider the sum of the simples $S=\bigoplus_{j=1}^t S_j$, the sum of the projective $\rmA$-resolutions $\scrQ_S=\bigoplus_{j=1}^t\scrQ^j$, and the sum $\PP=\bigoplus_{j=1}^t \PP^j$. Thus $S$ is finite dimensional, and there is a projective resolution $\scrQ_{S}\to S$ such that $\PP=i^*\scrQ_{S}$ satisfies 
\begin{equation}
\mathrm{H}^{-j}(\PP) = \left\{ \begin{array}{ll}
 S & \text{ if $j\in \{ 0, 3\}$ } \\
 0 & \text{ otherwise}.
\end{array}
\right.\label{eqn:flopslookabstract}
\end{equation}

\subsection{Single curve Donovan--Wemyss}
The case when there is a single curve, namely $t=1$, now follows easily since the $S$ in \eqref{eqn:flopslookabstract} is both finite dimensional and simple, and so \eqref{eqn:flopslookabstract} and the paragraphs above fall directly within the remit of Setup~\ref{Acon setup}.

\begin{cor}[{Conjecture \cite[1.4]{DW1}}]\label{main text local}
Suppose that $X_1\to\Spec R_1$ and $X_2\to\Spec R_2$ are two $3$-fold flopping contractions, where both $X_i$ are smooth, both $R_i$ are complete local, and both contractions have precisely one curve above the origin.  Write $\Acon$ and $\Bcon$ for their corresponding contraction algebras.  Then
\[
R_1\cong R_2\iff \Acon\cong\Bcon.
\]
\end{cor}
\begin{proof}
The implication (1)$\Rightarrow$(2) is contained in the original paper \cite{DW1}, the content is the direction (2)$\Rightarrow$(1).  In addition to notation already introduced, write $\rmA$ and $\mathrm{B}$ for the two NCCRs.  Since for both contractions the corresponding \eqref{eqn:flopslookabstract} falls within the remit of Setup~\ref{Acon setup}, we can simply quote \ref{cor:main_reconstruct} to conclude that there must be an $A_\infty$-quasi-isomorphism between  $\scrEnd_{\rmA}(\scrQ_{\rmA})$ and $\scrEnd_{\mathrm{B}}(\scrQ_{\mathrm{B}})$.

From here, the proof is standard.  As a consequence of the above, the Koszul duals $\scrEnd_{\rmA}(\scrQ_{\rmA})^!$ and $\scrEnd_{\mathrm{B}}(\scrQ_{\mathrm{B}})^!$ are $A_\infty$-quasi-isomorphic.  But these are the derived contraction algebras of Booth \cite{Booth} (see also \cite[2.6, 2.7]{KalckYang}), and so the derived contraction algebras are $\DG$-quasi-isomorphic (see e.g.\ \cite[2.8]{Lunts} or \cite[11.4.9]{LV}).  The result then follows directly from \cite[8.3.3]{Booth}, since a quasi-isomorphism between the derived contraction algebras implies that the relative and classical singularity categories are equivalent, and so we can use the recovery theorem of \cite[5.9]{HuaKeller}.
\end{proof}

\begin{remark}
Although the proof of \ref{main text local} does not split into cases, this is only since the two cases in the proof of \ref{lem.qisom-endomo-simple} both led to the same outcome.  In reality, there are two cases: when $\Acon$ is semisimple (namely $\Acon\cong\mathbb{C}$, the `Atiyah flop') and otherwise. In the Atiyah case the complex $\PP$ is
\[
\mathbb{C}\to 0\to 0 \to\mathbb{C}
\]
whereas in all other cases $\PP$ has four non-zero terms, and the morphisms all lie in the radical.  
Regardless, in both cases we can appeal to \ref{lem.qisom-endomo-simple}, and this is what makes a uniform proof of \ref{main text local} possible.
\end{remark}

The multi-curve version $t>1$ of the conjecture is proved in \ref{main text} below, but to do this requires some constructions to be extended so that they work over a semi-simple base. Indeed, when there is more than one curve, there is more than one simple, so the $S$ in \eqref{eqn:flopslookabstract} is semi-simple (not simple) and thus \ref{cor:main_reconstruct} is not directly applicable.  However, the extension needed in Section~\ref{sec:CatUpgrade} below is very mild, and the main ideas remain the same.

\section{Categorical Upgrade}\label{sec:CatUpgrade}

The purpose of this section is to first extend (in \S\ref{subsec:unitalposCat}) the notion of `unitally positive'  from $A_\infty$-algebras to $A_\infty$-categories, then (in \S\ref{subsec:GenIdem}) to use this to extend the results in \S\ref{sec:idems} to cover semi-simple modules.  Once this is done, in \S\ref{subsec:DWmulti} the multi-curve generalisation of  \ref{main text local} easily follows.

In this section, we remark that instead of the more widely used notion of quasi-equivalence of $A_\infty$-categories, we will instead use the notion of quasi-isomorphism, 
i.e., a quasi-equivalence  such that the map between the sets of objects  is a bijection, see e.g. \cite[1.13]{COS2}. We will mostly consider $A_\infty$-categories with a finite number of objects, and so the objects form a set. Thus we can appeal to \cite[Section 5.1]{LH}, and so the results of Section~\ref{sec:StrictlyUnital} will generalise easily, provided that when we consider strictly unital $A_\infty$-categories we require the property that for all $x \in \Ob \scrA $, the complex $\scrHom_\scrA(x,x)$ has non-trivial cohomology.

\subsection{Unitally positive categories}\label{subsec:unitalposCat}

\begin{defin}
   A strictly unital $A_\infty$-category is an  $A_\infty$-category $\scrA$ such that for every $x \in \Ob \scrA$ there exists a (unique) degree $0$ morphism $e_x\in \scrHom_{\scrA}(x, x) $, called a strict unit of $x$,
satisfying the following properties:
\begin{enumerate}
    \item $m_2(- \otimes e_x)= \Id_{\scrHom_{\scrA}(x, y)}$ and  $m_2(e_x \otimes -)= \Id_{\scrHom_{\scrA}(y, x)}$ for every $x,y \in \Ob \scrA$,
    \item $m_i(f_i \otimes \hdots \otimes f_1)=0$ if $i \neq 2$ and $f_j=e_x$ for some $j \in \{1, \hdots i\}$ and some $x \in \Ob \scrA$.
\end{enumerate}
\end{defin}
A $\DG$-category can be considered as a strictly unital $A_\infty$-algebra, because $\DG$-categories are required by definition to have units to ensure that they are actually categories.

An $A_\infty$-category is called cohomologically unital if $\HH^*(\scrA)$ is a category, in the sense that $\HH^*(\scrA)$ has units. A strictly unital $A_\infty$-category $\scrA$ is in particular cohomologically unital.

For an $A_\infty$-functor $F \colon \scrA \rightsquigarrow \scrB$ we denote by $F_0$ the morphism of sets $F_0 \colon \Ob \scrA \to \Ob \scrB$, and by $F_i$ ($i \geq 1$) the components of $F$. We write $\HH(F) \colon \HH^*(\scrA) \to \HH^*(\scrB) $ for the functor induced by $F_0, F_1$ between the cohomology categories. 
\begin{defin} Let $F \colon \scrA \rightsquigarrow \scrB$ be an $A_\infty$-functor.
\begin{enumerate}
    \item $F$ is cohomologically unital if $\scrA,\scrB$ are cohomologically unital  and $\HH(F)$ is unital.
    \item $F$ is strictly unital if $\scrA,\scrB$ are strictly unital and 
    \begin{itemize}
        \item $F_1(e_x)= e_{F_0(x)}$ for every $x \in \Ob \scrA$;
        \item $F_i(f_i \otimes \hdots \otimes f_1 )=0$ if $i \geq 2$ and $f_j = e_x$ for some $j \in \{1, \hdots, i\}$ and some $x \in \Ob \scrA$.
    \end{itemize}
    \item $F$ is a quasi-isomorphism if $F_0 \colon \Ob \scrA \to \Ob \scrB$ is a bijection, and for every $x,y \in \Ob \scrA$, $F_1 \colon \scrHom_{\scrA}(x,y) \to \scrHom_{\scrB}(F_0(x),F_0(y)) $ is a quasi-isomorphism. 
\end{enumerate}
\end{defin}

From now on, for reasons discussed in \cite{Seidel errata} and in \cite{COS2}, to ensure the existence of a strictly unital minimal model, we require our $\DG$- and $A_\infty$-categories to satisfy the following property.
\begin{property}\label{property-cats}
    For all objects $x$, the complex $\scrHom_{\scrA}(x, x)$ has non-trivial cohomology.
\end{property}
This property is clearly preserved by quasi-isomorphisms. 
Notice that if $\scrA$ is a $\DG$-category, or a strictly unital $A_\infty$-category, \ref{property-cats} is equivalent to requiring that for all objects $x$ the identity map $e_x$ is non-trivial in the cohomology of the complex $\scrHom_{\scrA}(x, x)$.

\begin{prop}\label{minimal cat}
Let $\scrA$ be a strictly unital $A_\infty$-category such that \textnormal{\ref{property-cats}} holds.  Then there exists a strictly unital minimal model for $\scrA$, i.e., a strictly unital  and minimal $A_\infty$-category $\scrA'$ equipped with two strictly unital $A_\infty$-quasi-isomorphisms $G \colon \scrA' \rightsquigarrow \scrA$ and $F\colon \scrA \rightsquigarrow \scrA'$.
\end{prop}
\begin{proof}
Following \cite[1.13]{Seidel}, for every pair $x,y$ of objects of $\scrA$ it is possible to choose a splitting of the complex $\scrHom_{\scrA}(x,y)$ into a minimal part and an acyclic part, which is isomorphic to $\HH^*(\scrHom_\scrA(x,y))$, and to construct a contracting homotopy $T_1$:
 \[
\begin{tikzpicture}[>=stealth]
\node (a) at (0,0) {$\HH^*(\scrHom_\scrA(x,y))$};
\node (b) at (5,0) {$\scrHom_{\scrA}(x,y)$};
\node (A) at (1.5,0) {$\phantom{1}$};
\node (B) at (3.8,0) {$\phantom{1}$};
\node (C) at (6.1,0) {$\phantom{1}$};
\draw[->,bend right,looseness=0.8] (A) to node[below]{$\scriptstyle G_1$}(B);
\draw[->,bend right,looseness=0.8] (B) to node[above]{$\scriptstyle F_1$}(A);
\draw[<-]  (C) edge [in=-35,out=35,loop,looseness=7] node[right]{$\scriptstyle T_1$}   (C);
\end{tikzpicture} 
\]
Furthermore, by the hypotheses on $\scrA$, it is possible to choose the splitting in such a way that when $x=y$ the chain maps $F_1$ and $G_1$ are unit-preserving. Then by \cite[1.12]{Seidel}, one can equip the category $\HH^*(\scrA)$ with an $A_\infty$-structure $\{\upmu_k\}_{k \geq 2}$ and extend $F_1$ and $G_1$ to $A_\infty$-quasi-isomorphisms $F \colon \scrA \rightsquigarrow (\HH^*(\scrA), \{\upmu_k\}_{k \geq 2})$ and $G \colon (\HH^*(\scrA), \{\upmu_k\}_{k \geq 2}) \rightsquigarrow \scrA$. Since $\scrA$ is strictly unital, $\HH^*(\scrA)$ is a cohomologically unital $A_\infty$-category, and it is easy to see that because of how we chose $F_1$ and $G_1$, the functors $F,G$ are cohomologically unital (for this, see also \cite[1.14]{COS2}).
		
The $A_\infty$-category $(\HH^*(\scrA), \{\upmu_k\}_{k \geq 2})$ satisfies the corrected hypotheses of \cite{Seidel errata}, so by  \cite[2.1]{Seidel} there is a formal diffeomorphism 
\[
\upphi \colon (\HH^*(\scrA), \{\upmu_k\}_{k \geq 2}) \rightsquigarrow (\HH^*(\scrA), \{\uplambda_k\}_{k \geq 2}),
\]
where $ \{\uplambda_k\}_{k \geq 2}$ is a strictly unital $A_\infty$-structure on  $\HH^*(\scrA)$, induced by $\upphi$. Note that by the proof of \cite[2.1]{Seidel}, $(\HH^*(\scrA), \{\uplambda_k\}_{k \geq 2})$ is still minimal. The formal diffeomorphism $\upphi$ above is such that its linear component is the identity, hence it is cohomologically unital.
		
It follows that there is a diagram of $A_\infty$-functors
\[ 
\scrA \stackrel{F}{\rightsquigarrow} (\HH^*(\scrA), \{\upmu_k\}_{k \geq 2}) \stackrel{\upphi}{\rightsquigarrow} (\HH^*(\scrA), \{\uplambda_k\}_{k \geq 2}),
\]
where both $F$ and $\upphi$ are cohomologically unital $A_\infty$-quasi-isomorphisms, both $\scrA$ and $(\HH^*(\scrA), \{\uplambda_k\}_{k \geq 2})$ are strictly unital $A_\infty$-categories, whilst $(\HH^*(\scrA), \{\upmu_k\}_{k \geq 2})$ is cohomologically unital. 

A cohomologically unital functor $F \colon \scrA \rightsquigarrow \scrB$ between strictly unital $A_\infty$-categories which satisfy \ref{property-cats} can be replaced with a strictly unital functor $F' \colon \scrA \rightsquigarrow \scrB$ homotopic to $F$, by \cite[2.2]{Seidel}, \cite[3.2.2.1, 3.2.4.3, Section 5.1]{LH}, or \cite[2.5]{COS} together with \cite[p. 6]{COS2}. Using this, we can construct a strictly unital $A_\infty$-quasi-isomorphism $F' \colon \scrA \rightsquigarrow \scrA'\colonequals (\HH^*(\scrA), \{\uplambda_k\}_{k \geq 2})$. The same can be done to construct a strictly unital $A_\infty$-quasi-isomorphism $G' \colon \scrA' \rightsquigarrow \scrA$.
	\end{proof}

\begin{construction}\label{const:N-cats}
Let $\scrA$ be a $\DG$-category  such that \ref{property-cats} holds.  Viewing $\scrA$ as a strictly unital $A_\infty$-category, consider its strictly unital minimal model  $\HH^*(\scrA)$, as in \ref{minimal cat}.
\begin{enumerate}
\item\label{Ndef1a} For objects $x \neq y$ in $\scrA$, consider the graded subspace
\[
\HH^{>0} (\scrHom_\scrA(x,y)) \subseteq \HH^{*} (\scrHom_\scrA(x,y)).
\]
\item\label{Ndef1b} For every object $x\in\scrA$, consider the graded subspace
\[
\mathbb{C}e_x \oplus\HH^{>0} (\scrHom_\scrA(x,x)) \subseteq \HH^{*} (\scrHom_\scrA(x,x)).
\]
\end{enumerate}
Together, \eqref{Ndef1a} and \eqref{Ndef1b} defines a vector subspace $\scrN_\scrA$ of $\HH^*(\scrA)$.
\end{construction}
\begin{lemma}\label{N_for_cats}
The strictly unital, minimal $A_\infty$-structure on $\HH^*(\scrA)$ restricts to a minimal $A_\infty$-structure on $\scrN_\scrA$ 
\end{lemma}
 \begin{proof}
This is exactly the same as \ref{lem.connective-truncation}, with the point being that $A_\infty$-products on inputs of degree $\geq 1$ all clearly land within $\scrN_\scrA$.  Products which involve degree zero input necessarily involve the units, which by inspection also always land in $\scrN_\scrA$.  
 \end{proof}

\begin{defin}
We call $\scrN_\scrA$ the \emph{unitally positive $A_\infty$-category} associated to $\scrA$.
\end{defin}

\begin{remark}\label{scrN-well-defined-cats}
 Similarly to \ref{scrN is well-defined}, we can also prove that $\scrN_\scrA$ is well-defined: if $\scrA$ and $\scrB$ are quasi-isomorphic 
$\DG$-categories such that \ref{property-cats} holds, we can replace  the zigzag of quasi-isomorphisms of $\DG$-categories between $\scrA$ and $\scrB$ with  a strictly unital $A_\infty$-quasi-isomorphism $F \colon \scrA \rightsquigarrow \scrB$.
We then obtain an $A_\infty$-quasi-isomorphism between their strictly unital minimal models constructed in \ref{minimal cat}, which by a simple degree argument, as in \ref{scrN is well-defined}, restricts to an $A_\infty$-quasi-isomorphism between
$\scrN_\scrA$ and $\scrN_\scrB$.
\end{remark}

\begin{lemma}\label{ST-cats}
Let $E_i, E'_i$ be bounded complexes of projective $\Upgamma$-modules, for $i \in \{1, \hdots, m\}$, with quasi-isomorphisms $g_i \colon E_i \to E'_i$ for every $i$. Then the $\DG$-categories
$(i,j)\mapsto \scrHom_\Upgamma(E_i, E_j)$ and $(i,j)\mapsto \scrHom_\Upgamma(E'_i, E'_j)$ are quasi-isomorphic. 
\end{lemma}

\begin{proof}
    This is entirely based on \cite[4.4]{ST}, but for completeness we spell out how it works for $\DG$-categories with $m$ objects. 

    For every $i$, denote $C_i \colonequals \operatorname{cone}(g_i)= (E_i[1]\oplus E_i', d_{C_i})$, and define the $\DG$-category with $m$ objects:
    \[ (i,j) \mapsto \scrHom_\Upgamma(C_i, C_j)=  \left( \begin{array}{cc}
		\scrHom_\Upgamma(E_i[1],E_j[1]) & \scrHom_\Upgamma(E_i',E_j[1])\\
		\scrHom_\Upgamma(E_i[1],E_j')  & \scrHom_\Upgamma(E_i',E_j')
	\end{array} \right)\]
with obvious differential, composition and identities. 
The calculation of \cite[4.4]{ST} shows that lower triangular matrices are closed for the differentials, so that
 \[ (i,j) \mapsto  \left( \begin{array}{cc}
		\scrHom_\Upgamma(E_i[1],E_j[1]) &0\\
		\scrHom_\Upgamma(E_i[1],E_j')  & \scrHom_\Upgamma(E_i',E_j')
	\end{array} \right)\]
with same differentials, composition and identities, is also a $\DG$-category with $m$ objects. Now following the rest of the proof of \cite[4.4]{ST}, projections onto the categories 
\[ (i,j) \mapsto  \scrHom_\Upgamma(E_i[1],E_j[1]) \cong \scrHom_\Upgamma(E_i,E_j)\]
and  \[ (i,j) \mapsto \scrHom_\Upgamma(E_i',E_j') \] give a roof of quasi-isomorphisms of $\DG$-categories. 
\end{proof}
\subsection{The case of idempotents}\label{subsec:GenIdem}

Let $\rmA, \Acon, i^*$ and $i_*$ be as in Section~\ref{sec:idems}.

\begin{setup}\label{setup-multi}
    Let $S_1, \hdots, S_t \in\mod\rmA$ be finite dimensional simple modules, and suppose that there exists for every $j \in \{1, \hdots, t \}$ a complex of finitely generated projective $\rmA$-modules
\[
\scrQ^j\colonequals\quad  0 \rightarrow Q^j_{n-1} \rightarrow Q^j_{n-2} \rightarrow \dots \rightarrow Q^j_{1} \rightarrow Q^j_0 \rightarrow 0
\]
which gives an $\rmA$-projective resolution of $S_j$, such that $\PP^j\colonequals i^*\scrQ ^j$ satisfies
\[
\mathrm{H}^{-k}(\PP^j) \cong \left\{ \begin{array}{ll}
 i^*S_j & \text{ if $k\in \{ 0, n-1\}$ } \\
 0 & \text{ otherwise},
\end{array}
\right.
\]
and hence is a length $n$ periodic projective resolution of $S_j$.
\end{setup}

Under Setup~\ref{setup-multi}, set $\scrA$ to be the $\DG$-category with $t$ objects $\{1, \hdots, t\}$ defined as
\[ 
(i,j) \mapsto \scrHom_{\scrA}(i,j)\colonequals\scrHom_{\Acon} (\PP^i, \PP^j)
\]
with the obvious composition and identities. 
Considering $S_1, \hdots, S_t$  as $\rmA$-modules, we can also consider the $\DG$-category $\scrB$, also with $t$ objects, defined as
\[ 
(i,j) \mapsto \scrHom_{\scrB}(i,j)\colonequals\scrHom_{\rmA} (\scrQ^i, \scrQ^j).
\]

If we choose another set of $\rmA$-projective resolutions $\scrQ'^{j}\to S_j$, it is clear that we obtain another $\DG$-category $\scrB'$ which is quasi-isomorphic to $\scrB$, by \ref{ST-cats}. Checking that $\scrA$ is well-defined up to quasi-isomorphism is the same, with one extra step: it relies on the fact that each $S_j$ is simple, so by \ref{lem.qisom-endomo-simple}  $\PP^j$ and $\PP'^j$ are quasi-isomorphic, and so we can also use \ref{ST-cats}. From this, by \ref{scrN-well-defined-cats} the $A_\infty$-quasi-isomorphism class of the $A_\infty$-category $\scrN_\scrA$ is well defined.

\begin{thm}\label{Main DG result cats}
Under Setup~\textnormal{\ref{setup-multi}}, and with notation as above, there is an $A_\infty$-quasi-isomorphism between the $\DG$-category $\scrB$ and the unitally positive $A_\infty$-category $\scrN_{\scrA}$ associated to $\scrA$. 
\end{thm}
\begin{proof} 
 Since for every $j$, $\scrQ^j \to S_j$ is an $\rmA$-projective resolution, we have
 \[ \scrHom_{\HH^*(\scrB)}(i,j) = \HH^*(\scrHom_\scrB(i,j)) =\HH^* (\scrHom_{\rmA} (\scrQ^i, \scrQ^j))\cong \Ext^*_{\rmA}(S_i, S_j).\]
 For the $\DG$-category $\scrA$, we can adapt the cohomology calculations of Section~\ref{sec.cohomology_end} to see that for $i\neq j$, the only non-zero cohomology for $\scrHom_{\Acon}(\PP^i, \PP^j)$ is
 \begin{equation}\label{eq.coh-ext-mixedI}
 \HH^{k}(\scrHom_{\Acon}(\PP^i, \PP^j)) = \begin{cases} 
 \Ext^{k+n-1}_{\Acon} (S_i, S_j) &\quad k=1-n, \hdots, 0\\
 \Ext^i_{\Acon}(S_i,S_j) &\quad k=1,  \hdots , n-1 \\ 
  \end{cases}
 \end{equation}
 whilst for $i = j$, the only non-zero cohomology is
 \begin{equation}\label{eq.coh-ext-mixedII}
 \HH^{k}(\scrHom_{\Acon}(\PP^j, \PP^j)) = \begin{cases} 
 \Ext^{k+n-1}_{\Acon} (S_j, S_j) &\quad  k=1-n, \hdots, -1\\
 \Ext^{n-1}_{\Acon} (S_j, S_j)\oplus \mathbb{C} \oplus 
 \image (\circ d_0^j) &\quad k=0 \\
 \Ext^k_{\Acon}(S_j,S_j) &\quad k=1,  \hdots , n-2 \\ \Ext^{n-1}_{\Acon}(S_j,S_j) \oplus \image(\circ d_0^j) &\quad k =n-1,
  \end{cases}
 \end{equation}
 where we are using the fact that $\Hom_{\Acon}(S_j,S_j)\cong \mathbb{C}$, $\Hom_{\Acon}(S_i,S_j)=0$ for all $i \neq j$, and the notation $\image(\circ d_0^j)$ refers to the image of the map
 \[
 \circ d_0^j \colon \Hom_{\Acon} (P_{n-1}^j, S_j) \to \Hom_{\Acon}(P_0^j, S_j)
 \]  
 considered in \ref{lem.map-pi}. Thus the graded components of the unitally positive $A_\infty$-category $\scrN_{\scrA}$ associated to $\scrA$ of \ref{const:N-cats} are
 \begin{equation}\label{eq.N-catsI}
 		\scrHom_{\scrN_{\scrA}}^k(j,j) = \begin{cases} 0 &\quad k \leq -1,\ k \geq n\\ \mathbb{C}&\quad k = 0\\
 			\Ext^k_{\Acon}(S_j,S_j) &\quad k=1, \hdots , n-2 \\ \Ext^{n-1}_{\Acon}(S_j,S_j) \oplus \image(\circ d_0^j) &\quad k =n-1
 			. \end{cases}
 \end{equation}
 whilst for $i \neq j$ the graded components are
  \begin{equation}\label{eq.N-catsII}
      \scrHom_{\scrN_{\scrA}}^k(i,j) =\begin{cases} 0 &\quad k \leq 0,\ i \geq n\\ 
 			\Ext^i_{\Acon}(S_i,S_j) &\quad k=1, \hdots , n-1 
    			. \end{cases}
  \end{equation}
 Extension of scalars gives a morphism of $\DG$-categories $i^* \colon \scrB \to\scrA$ sending
 \[ 
 f \in \scrHom_{A} (\scrQ^i, \scrQ^j) =\scrHom_{\scrB}(i,j)\mapsto i^*f \in \scrHom_{\Acon} (\PP^i, \PP^j)=  \scrHom_{\scrA}(i,j)
 \]
 which is the identity on objects.  The following commutative diagram is a generalisation of \eqref{eq.diagram}
 \begin{equation}\label{eq.diagram2}
 		\begin{tikzcd}
 			\scrHom_\rmA(\scrQ^i, \scrQ^j) \arrow[rd, "\upvarepsilon_j\circ "', bend right=15] \arrow[r, "i^*"] & \scrHom_{\Acon}(\PP^i, \PP^j) \arrow[r, "\upiota_j\circ "]        & {\scrHom_{\Acon}(\PP^i,\scrP^j)} \arrow[ld, "\upalpha_j\circ", bend left=15] \\ & {\scrHom_\rmA(\scrQ^i, S_j)\cong \scrHom_{\Acon}(\PP^i,S_j)}, &                                                             
 		\end{tikzcd}
 	\end{equation}
 and the maps $\upvarepsilon_j \circ$, $\upalpha_j\circ$ are still quasi-isomorphisms. This implies 
 \begin{enumerate}
     \item The morphism of $\DG$-categories $ i^* \colon \scrB \to\scrA$ is injective in cohomology.
     \item The cohomology of $\scrB$ can be calculated explictly, by adapting the cohomology calculation in \S\ref{sec.cohomology_end}, namely
 \begin{align*}
 \Ext^k_\rmA(S_j, S_j)&\cong \HH^k(\scrHom_\rmA(\scrQ^j, \scrQ^j))   \nonumber\\
 &\cong \HH^k(\scrHom_{\Acon}(\PP^j, \scrP^j))\\
&= \begin{cases} 0 &\quad k < 0,\ k \geq n\\ 
 			\Ext^k_{\Acon}(S_j,S_j) &\quad k=0,  \hdots , n-2 \\ \Ext^{n-1}_{\Acon}(S_j,S_j) \oplus \image(\circ d_0^j) &\quad i =n-1
 		\end{cases}
\end{align*}
and if $i\neq j$ then 
 \begin{align*}
   	\Ext^k_\rmA(S_i, S_j)&\cong \HH^k(\scrHom_\rmA(\scrQ^i, \scrQ^j))   \nonumber\\
 		&\cong \HH^k(\scrHom_{\Acon}(\PP^i, \scrP^j))\\
		& = \begin{cases} 0 &\quad k \leq 0,\ k \geq n\\ 
 			\Ext^k_{\Acon}(S_i,S_j) &\quad k=1,  \hdots , n-1.
 		\end{cases}
 	\end{align*}
 \end{enumerate}
 By comparing this with \eqref{eq.N-catsI} and \eqref{eq.N-catsII}, we can see that for every $i, j$, $\Ext^*_{\Acon}(S_i,S_j)$ and  $\scrHom_{\scrN_{\scrA}}(i,j)$  are isomorphic as graded vector spaces. 

 Passing to strictly unital minimal models defined in \ref{minimal cat}, the morphism of $\DG$-categories $ i^* \colon \scrB \to\scrA$ induces a strictly unital $A_\infty$-morphism of strictly unital $A_\infty$-categories $F \colon \HH^*(\scrB) \rightsquigarrow \HH^*(\scrA)$, such that for every $i,j$ its linear component
 \[F_1 \colon \scrHom_{\HH^*(\scrB)}(i,j) \to \scrHom_{\HH^*(\scrA)}(i,j) \] is injective. By the same argument of \ref{Main DG result}, $F$ restricts to an $A_\infty$-morphism  $F \colon \HH^*(\scrB) \rightsquigarrow \scrN_\scrA$, such that its linear component is injective.  Counting dimensions, the linear component is thus an isomorphism, proving the claim.   
\end{proof}

The following is the analogue of \ref{cor:main_reconstruct}.

 \begin{cor}\label{cor:main_reconstruct_cats}
Suppose that both $\rmA$ and $\mathrm{B}$ admit $t$ finite dimensional simples, and idempotents which satisfy the conditions in Setup~\textnormal{\ref{setup-multi}}, via complexes $\scrQ^j_{\rmA}$ and $\scrQ^j_{\mathrm{B}}$ of the same length, for each $j \in \{1, \hdots, t\}$.  If $\Acon\cong\Bcon$, and each has precisely $t$ simple modules, then there is an $A_\infty$-quasi-isomorphism between the $\DG$-categories $(i,j) \mapsto \scrHom_{\rmA}(\scrQ^i_{\rmA}, \scrQ^j_{\rmA})$ and $(i,j) \mapsto \scrHom_{\mathrm{B}}(\scrQ^i_{\mathrm{B}}, \scrQ^j_{\mathrm{B}})$.
\end{cor}
\begin{proof}
The simple $\Acon$-modules $S_j$ are equipped with length $n$ complexes $\PP^j$ which build $\scrN_{\rmA}$, and by \ref{Main DG result cats} there is an $A_\infty$-quasi-isomorphism
between the $\DG$-category $(i,j) \mapsto \scrHom_{\rmA}(\scrQ^i_{\rmA}, \scrQ^j_{\rmA})$ and the $A_\infty$-category $\scrN_{\rmA}$, equipped with the $A_\infty$-structure of \ref{N_for_cats}.
 The same holds for $\mathrm{B}$, where $\Bcon$ has $t$ simples $S'^j$, equipped with complexes $\PP'^j$ of same length $n$, which in turn build $\scrN_{\mathrm{B}}$. Again by \ref{Main DG result cats}, there is an $A_\infty$-quasi-isomorphism between the $\DG$-category $(i,j) \mapsto \scrHom_{\mathrm{B}}(\scrQ^i_{\mathrm{B}}, \scrQ^j_{\mathrm{B}})$ and the $A_\infty$-category $\scrN_{\mathrm{B}}$, equipped with the $A_\infty$-structure of \ref{N_for_cats}.
 
It now suffices to prove that $\scrN_\rmA$ and $\scrN_{\mathrm{B}}$ are $A_\infty$-quasi-isomorphic, and for this we will ultimately use the fact that the unitally positive category is well-defined up to quasi-isomorphism.
Indeed, since $\Acon\cong\Bcon$, there is an (exact) isomorphism of categories $F\colon \mod\Acon\to\mod\Bcon$.  Now $F$ preserves simples, so without loss of generality we can set $F S^j\cong S'^j$.  Further, $F$ takes $\PP^j$ to some length $n$ complex of projectives which builds a periodic projective resolution of $S'^j$. Clearly this isomorphism of module categories induces an isomorphism of $\DG$-categories 
\begin{equation}
[(i,j) \mapsto \scrHom_{\Acon}(\PP^i, \PP^j)] \xrightarrow{\cong }\ [(i,j) \mapsto \scrHom_{\Bcon}(F\PP^i, F\PP^j)].\label{eqn:iso2}
\end{equation}

Now for every $j$, $F\PP^j$ is quasi-isomorphic to $\PP'^j$ by \ref{lem.qisom-endomo-simple}, as both give periodic resolutions of the same simple with the same length. By \ref{ST-cats}, the $\DG$-categories $(i,j)  \mapsto \scrHom_{\Bcon}(F\PP^i, F\PP^j) $ and $(i,j) \mapsto \scrHom_{\Bcon}(\PP'^i, \PP'^j) $ are then quasi-isomorphic.  

Combining this with
\eqref{eqn:iso2}, it follows that the $\DG$-categories $(i,j) \mapsto \scrHom_{\Acon}(\PP^i, \PP^j)$ and $(i,j) \mapsto \scrHom_{\Bcon}(\PP'^i,\PP'^j)$ are quasi-isomorphic.  By \ref{scrN-well-defined-cats}, $\scrN_\rmA$ and $\scrN_{\mathrm{B}}$ are $A_\infty$-quasi-isomorphic, proving the result.
\end{proof}

\subsection{General Donovan--Wemyss}\label{subsec:DWmulti}
We revert to the general flops setup of \S\ref{subsec:flopssetup}, where there are $t$ curves $C_1,\hdots,C_t$.
\begin{thm}\label{main text}
Suppose that $X_1\to\Spec R_1$ and $X_2\to\Spec R_2$ are two $3$-fold flopping contractions, where both $X_i$ are smooth, and both $R_i$ are complete local.  Denote their corresponding contraction algebras by $\Acon$ and $\Ccon$, respectively.  Then the following conditions are equivalent.
\begin{enumerate}
\item $R_1\cong R_2$.
\item $\Db(\mod\Acon)\simeq\Db(\mod\Ccon)$.
\item $\Acon$ is isomorphic to an iterated mutation of $\Ccon$.
\item $\Acon$ is isomorphic to $\Bcon$, for some other crepant resolution $Y_2\to\Spec R_2$.
\end{enumerate}
\end{thm}
\begin{proof}
(1)$\Rightarrow$(2)$\Leftrightarrow$(3)$\Leftrightarrow$(4) can all be found in the PhD thesis of August \cite[1.4, 1.5]{August}.  The content in the theorem is the direction (4)$\Rightarrow$(1).   As in \ref{main text local},  write $\rmA$ and $\mathrm{B}$ for the NCCRs corresponding to the two contractions.  For both contractions, \eqref{eqn:flopslookabstract} now falls within the remit of Setup~\ref{setup-multi}, we and so can simply quote \ref{cor:main_reconstruct_cats} to conclude that there must be an $A_\infty$-quasi-isomorphism between the $\DG$-categories $(i,j)\mapsto \scrHom_{\rmA}(\scrQ^i_{\rmA}, \scrQ^j_{\rmA})$ and $(i,j) \mapsto \scrHom_{\mathrm{B}}(\scrQ^i_{\mathrm{B}},\scrQ^j_{\mathrm{B}})$.  

From here, the proof is similar to that of~\ref{main text local}: as a consequence of the above, $\scrEnd_{\rmA}(\oplus\scrQ_{\rmA}^j)$ and $\scrEnd_{\mathrm{B}}(\oplus\scrQ_{\mathrm{B}}^j)$ are quasi-isomorphic, when viewed as $\DG$-algebras over $K=\mathbb{C}\times\hdots\times\mathbb{C}$. Thus taking their Koszul dual over $K$, it follows that $\Upgamma_\rmA\colonequals\scrEnd_{\rmA}(\scrQ_{\rmA})^!$ and $\Upgamma_{\mathrm{B}}\colonequals\scrEnd_{\mathrm{B}}(\scrQ_{\mathrm{B}})^!$ are $A_\infty$-quasi-isomorphic.  But these are the derived deformation algebras of Hua-Keller \cite[2.2]{HuaKeller} (see also \cite[2.6, 2.7]{KalckYang}), and so the derived deformation algebras are $\DG$-quasi-isomorphic (see e.g.\ \cite[2.8]{Lunts}).  

We now simply follow the final five lines of the proof of \cite[5.11]{HuaKeller}, as the quasi-isomorphism between the derived deformation algebras implies that there is an algebraic equivalence between the categories $per(\Upgamma_\rmA)$ and $per(\Upgamma_{\mathrm{B}})$, and thus between their subcategories with finite dimensional cohomology, since those can be characterised intrinsically.  It follows that their quotients, which are called the cluster categories in \cite{HuaKeller}, are also algebraically equivalent, and so by \cite[5.12]{HuaKeller} there is an induced algebraic triangle equivalence between the singularity categories $\mathrm{D}_{\mathsf{sg}}(R_1)$ and $\mathrm{D}_{\mathsf{sg}}(R_2)$.  Thus we obtain an algebra isomorphism $\mathrm{HH}^0(\mathrm{D}_{\mathsf{sg}}(R_1))\cong\mathrm{HH}^0(\mathrm{D}_{\mathsf{sg}}(R_2))$, and so by the recovery theorem \cite[5.9]{HuaKeller}, necessarily $R_1\cong R_2$.
\end{proof}

\end{document}